\documentclass[11pt]{amsart}
\usepackage[a4paper,hmargin={2.5cm,2.5cm},vmargin={2.5cm,2.5cm}]{geometry}

\usepackage[utf8]{inputenc}
\usepackage{lipsum}
\usepackage[T1]{fontenc}
\usepackage{lmodern}
\usepackage{mathrsfs}
\usepackage{xcolor} 
\usepackage{tikz}
\usepackage{booktabs}
\usetikzlibrary{calc}
\usepackage[T1]{fontenc}
\usepackage{pgfplots}
\pgfplotsset{compat=1.18}
\usepackage[english]{babel}
\usepackage{csquotes}
\usepackage{hyperref}
\usepackage[shortlabels]{enumitem}
\usepackage{libertinus}
\usepackage{libertinust1math}
\usepackage{cite}

\usepackage{mathtools}
\mathtoolsset{showonlyrefs}

\title{Least energy solutions for a quasilinear Neumann problem}
\date{\today}
\author{Annamaria Canino}
\address{ Dipartimento di Matematica e Informatica, Università della Calabria, Ponte Pietro Bucci cubo 31B,   87036 ARCAVACATA DI RENDE, COSENZA, ITALY}
\email{annamaria.canino@unical.it}
\author{Simone Mauro}
\address{ Dipartimento di Matematica e Informatica, Università della Calabria, Ponte Pietro Bucci cubo 31B,   87036 ARCAVACATA DI RENDE, COSENZA, ITALY}
\email{simone.mauro@unical.it}
\thanks{Both authors: Dipartimento di Matematica e Informatica, Universit\`a della Calabria,
Ponte Pietro Bucci cubo 31B, 87036 Arcavacata di Rende, Cosenza, Italy.
E-mail addresses: \texttt{annamaria.canino@unical.it} (A.~Canino) and
\texttt{simone.mauro@unical.it} (S.~Mauro).}   

\numberwithin{equation}{section}

\usepackage[leqno]{amsmath}
\usepackage{amssymb,amsthm}

\usepackage{tikz-cd}
\usepackage[pagewise]{lineno} 

\theoremstyle{plain}
\newtheorem{theorem}{Theorem}[section]
\newtheorem{lemma}[theorem]{Lemma}
\newtheorem{proposition}[theorem]{Proposition}
\newtheorem{corollary}[theorem]{Corollary}

\theoremstyle{plain}
\newtheorem{definition}[theorem]{Definition}

\newtheorem{example}{Example}[section]

\theoremstyle{plain}
\newtheorem{remark}[theorem]{Remark}

\RequirePackage{dsfont}

\newcommand{\R}{\field{R}}

\newcommand{\field}[1]{\mathbb{#1}}
\newcommand{\wto}{\rightharpoonup}

\tikzset{%
    symbol/.style={%
        draw=none,
        every to/.append style={%
            edge node={node [sloped, allow upside down, auto=false]{$#1$}}}
    }
}
\usepackage{enumitem}

\tikzset{shorten <>/.style={shorten >=#1,shorten <=#1}}

\begin{document}

\begin{abstract}
We study the existence of nonconstant weak solutions and the qualitative behavior of the least energy level for a class of quasilinear elliptic problems with Neumann boundary conditions of the form
\[
-\sum_{i,j=1}^N D_{x_j}(a_{ij}(x,u)D_{x_i}u)
+\frac12\sum_{i,j=1}^N D_s a_{ij}(x,u)D_{x_i}uD_{x_j}u
=g(u)-\lambda u \quad \text{in }\Omega.
\]
Since the associated energy functional is neither of class $C^1$ nor locally Lipschitz, the classical Nehari manifold method cannot be applied directly. We therefore rely on variational techniques combined with nonsmooth critical point theory.

For $\lambda>0$, we prove the existence of least
energy solutions and show that, under suitable
assumptions, they are nonconstant for large
$\lambda$ and constant for sufficiently small
$\lambda$. Moreover, the least energy level
tends to $+\infty$ as $\lambda\to+\infty$.

In the indefinite regime, a key ingredient is a spectral decomposition of $H^1(\Omega)$ associated with the linear operator
\[
\mathcal L_0u
:=-\sum_{i,j=1}^N
D_{x_j}\bigl(a_{ij}(x,0)D_{x_i}u\bigr).
\]

\end{abstract}



\maketitle
\section{Introduction}

In this paper we study the existence and asymptotic behavior of least energy solutions for a class of quasilinear elliptic equations with Neumann boundary conditions.

Our problem can be viewed as a generalization of the classical semilinear Neumann problem
\[
-\Delta u + \lambda u = |u|^{p-2}u \quad \text{in } \Omega, \qquad u>0\quad\text{in $\Omega$} \qquad \frac{\partial u}{\partial\nu} = 0 \quad \text{on } \partial\Omega,
\]
which has been widely investigated in the literature. In particular, it is known that the behavior of solutions strongly depends on the parameter $\lambda>0$: for sufficiently small positive $\lambda$, nonconstant positive solutions are excluded under suitable assumptions, whereas for large $\lambda$ nonconstant positive solutions appear and may exhibit concentration phenomena \cite{lin1988large,lin2006diffusion}.

Most of the available results concern positive solutions. For instance, in \cite{ni1991shape,ni1993peaks}, the authors study positive solutions and show that least energy solutions may concentrate near boundary points with maximal curvature.

{
In chemotaxis models, which describe the directed movement and
aggregation of cells in response to chemical signals, the stationary
equations can be reduced, under suitable assumptions, to a semilinear
Neumann problem of the form considered above. Nonconstant positive
solutions then correspond to spatially inhomogeneous stationary states;
see Lin, Ni and Takagi \cite{lin1988large}.

A further connection comes from the nonlinear Schr\"odinger equation
\[
-i\partial_t\Psi-\Delta\Psi=|\Psi|^{p-2}\Psi.
\]
Indeed, standing waves of the form
$\Psi(t,x)=e^{i\lambda t}u(x)$ have profiles satisfying
$-\Delta u+\lambda u=|u|^{p-2}u$, with the corresponding homogeneous
Neumann boundary condition. The nonlinear Schr\"odinger framework arises in nonlinear
optics, where it describes optical solitons
\cite{kivshar2003optical}, and, in the Gross--Pitaevskii
formulation, in models of Bose--Einstein condensates
\cite{pitaevskii2016bose}.
}

Our goal is to extend the study of least energy solutions to a quasilinear setting, where the diffusion operator depends on the solution itself, without imposing any sign condition on the solutions. In this framework, we show that the interplay between the parameter $\lambda$ and the nonlinear structure of the operator still governs the existence and qualitative properties of solutions, yielding a behavior that parallels, in part, the semilinear case.

\subsection{Problem setting and variational framework}

Let $\lambda \in \mathbb{R}$, $p \in \left(2,\frac{2N}{N-2}\right)$, and let $\Omega \subset \mathbb{R}^N$ be a bounded domain with $N \ge 3$ and Lipschitz boundary $\partial\Omega$.

The results presented below remain valid also in the cases $N=1,2$, where the condition on $p$ reduces to $p>2$.

We consider the following problem
\begin{equation}\tag{$\mathcal P_\lambda$}\label{P}
\left\{
\begin{aligned}
-&\sum_{i,j=1}^ND_{x_j}(a_{ij}(x,u)D_{x_i}u)+\frac12\sum_{i,j=1}^ND_sa_{ij}(x,u)D_{x_i}uD_{x_j}u=g(u)-\lambda u\quad \text{in $\Omega$},\\
&\sum_{i,j=1}^Na_{ij}(x,u)D_{x_i}u\ \nu_j=0\quad\text{on $\partial\Omega$},
\end{aligned}
\right.
\end{equation}
where $\nu=(\nu_1,\dots,\nu_N)$ is the outward normal, $a_{ij}:\Omega\times\R\to\R$ is a $C^1$-Carathéodory function, and $g:\R\to\R$ is a continuous function. 
{
The variational and compactness arguments can also be adapted to
nonlinearities of the form $g=g(x,s)$.

The problem with homogeneous Dirichlet boundary conditions is
simpler:
nonzero constant functions are excluded by the boundary condition.
}
We assume that the coefficients $a_{ij}:\Omega\times\mathbb{R}\to\mathbb{R}$ are measurable in $x$ and continuously differentiable in $s$, and that there exist constants $C_0,\nu_0>0$, $q>2$, and $\gamma\in(0,q-2)$ such that, for a.e. $x\in\Omega$, every $s\in\mathbb{R}$, and every $\xi\in\mathbb{R}^N$, the following properties hold:
\begin{align}
\tag{$a.1$}\label{a.1}
    |a_{ij}(x,s)|,\ |D_s a_{ij}(x,s)|\le C_0, \qquad a_{ij}(x,s)=a_{ji}(x,s),
\end{align}
\begin{align}
\tag{$a.2$}\label{a.2}
    \sum_{i,j=1}^N a_{ij}(x,s)\xi_i \xi_j \ge \nu_0 |\xi|^2,
\end{align}
\begin{align}
\tag{$a.3$}\label{a.3}
    0 \le \sum_{i,j=1}^N \left( D_s a_{ij}(x,s)\xi_i \xi_j \right)s
    \le \gamma \sum_{i,j=1}^N a_{ij}(x,s)\xi_i \xi_j.
\end{align}

\begin{example}
Let $\varepsilon>0$ and define, for $(x,s)\in \Omega\times\mathbb{R}$,
\[
a_{ij}(x,s)=A(s)\,\delta_{ij},
\qquad 
A(s):=1+\varepsilon\frac{s^{2}}{1+s^{2}}
=\frac{1+(1+\varepsilon)s^2}{1+s^2}.
\]

Then $a_{ij}(x,s)=a_{ji}(x,s)$ and
\[
1\le A(s)<1+\varepsilon
\qquad \text{for every }s\in\mathbb{R}.
\]

\medskip

\noindent
\textbf{Verification of \eqref{a.1}.}
We have
\[
|a_{ij}(x,s)|\le (1+\varepsilon)\delta_{ij}.
\]
Moreover,
\[
A'(s)=\frac{2\varepsilon s}{(1+s^2)^2},
\]
and therefore
\[
|A'(s)|
\le
2\varepsilon \max_{t\ge 0}\frac{t}{(1+t^2)^2}
=
\frac{3\sqrt{3}}{8}\,\varepsilon.
\]
Hence
\[
|D_s a_{ij}(x,s)|
=
|A'(s)|\,\delta_{ij}
\le C_0,
\]
with
\[
C_0=\max\left\{1+\varepsilon,\frac{3\sqrt{3}}{8}\,\varepsilon\right\}.
\]

\medskip

\noindent
\textbf{Verification of \eqref{a.2}.}
For every $\xi\in\mathbb{R}^N$,
\[
\sum_{i,j=1}^N a_{ij}(x,s)\xi_i\xi_j
=
A(s)|\xi|^2
\ge
|\xi|^2.
\]
Thus \eqref{a.2} holds with $\nu_0=1$.

\medskip

\noindent
\textbf{Verification of \eqref{a.3}.}
We compute
\[
\sum_{i,j=1}^N \left(D_s a_{ij}(x,s)\xi_i\xi_j\right)s
=
A'(s)s\,|\xi|^2,
\]
and
\[
A'(s)s
=
\frac{2\varepsilon s^2}{(1+s^2)^2}
\ge 0.
\]
Moreover,
\[
\frac{A'(s)s}{A(s)}
=
\frac{2\varepsilon s^2}{(1+s^2)(1+(1+\varepsilon)s^2)}
\le
\frac{2\varepsilon s^2}{1+s^2}
\le
2\varepsilon.
\]
Hence
\[
0
\le
\sum_{i,j=1}^N \left(D_s a_{ij}(x,s)\xi_i\xi_j\right)s
\le
2\varepsilon \sum_{i,j=1}^N a_{ij}(x,s)\xi_i\xi_j.
\]

Therefore \eqref{a.3} holds with $\gamma\ge2\varepsilon$. In particular, given any $\gamma>0$, it is enough to choose $\varepsilon\le \gamma/2$.

\end{example}

{
The coefficient introduced in the preceding example also appears
in quasilinear modified Schr\"odinger equations arising, for
instance, in models of the self-channeling of high-power ultrashort
laser pulses \cite{shenwang2013}. Consider
\[
-i\partial_t\Psi-\Delta\Psi
-\varepsilon\,\frac{\Psi}{\sqrt{1+|\Psi|^2}}
\Delta\sqrt{1+|\Psi|^2}
=|\Psi|^{p-2}\Psi,
\]
with homogeneous Neumann boundary conditions.
Looking for standing waves of the form
$\Psi(t,x)=e^{i\lambda t}u(x)$, with $u$ real-valued, yields
\[
-\operatorname{div}(A(u)\nabla u)
+\frac12 A'(u)|\nabla u|^2+\lambda u
=|u|^{p-2}u,
\qquad
A(s)=1+\varepsilon\frac{s^2}{1+s^2}.
\]
Thus, the stationary equation is a special case of
\eqref{P}, with $g(s)=|s|^{p-2}s$ and
$a_{ij}(x,s)=A(s)\delta_{ij}$.
}

Concerning the nonlinearity $g:\mathbb{R}\to\mathbb{R}$, we assume that it is continuous and satisfies
\begin{align}
\tag{$g.1$}\label{g.1'}
    |g(s)|\le C_0(1+|s|^{p-1}), \qquad \lim_{s\to0}\frac{g(s)}{s}=0,
\end{align}
and
\begin{align}
\tag{$g.2$}\label{g.2}
    g(s)s \ge q G(s) > 0 \quad \text{for all } s\neq 0,
\end{align}
where $G(s)=\displaystyle\int_0^sg(t)\,dt$.

{
Throughout the paper we also assume
\begin{equation}\label{symmetry}
a_{ij}(x,-s)=a_{ij}(x,s),\qquad g(-s)=-g(s)
\end{equation}
for a.e. $x\in\Omega$, every $s\in\mathbb R$, and all $i,j=1,\dots,N$. 
}

In particular, condition \eqref{g.1'} implies that, for every $\varepsilon>0$, there exists $C_\varepsilon>0$ such that
\begin{align}
\label{g.1}
    |g(s)| \le \varepsilon |s| + C_\varepsilon |s|^{p-1}.
\end{align}

We say that $u\in H^1(\Omega)$ is a weak solution of \eqref{P} if
\[
\int_\Omega \sum_{i,j=1}^Na_{ij}(x,u)D_{x_i}uD_{x_j}v+\frac12\int_\Omega\sum_{i,j=1}^N\left(D_sa_{ij}(x,u)D_{x_i}uD_{x_j}u\right) v=\int_\Omega g(u)v-\lambda\int_\Omega uv,
\]
for every $v\in H^1(\Omega)\cap L^\infty(\Omega)$.

We define the energy functional $J_\lambda:H^1(\Omega)\to\R$ as follows
\[
J_\lambda(u)=\frac12\int_\Omega \sum_{i,j=1}^Na_{ij}(x,u)D_{x_i}uD_{x_j}u+\frac{\lambda}{2}\int_\Omega u^2-\int_\Omega G(u),
\]

The functional $J_\lambda$ is continuous on
$H^1(\Omega)$, but in general it is neither
of class $C^1$ nor locally Lipschitz.
Indeed, for every $u\in H^1(\Omega)$,
the functional $J_\lambda$ is G\^ateaux
differentiable along every direction
$v\in H^1(\Omega)\cap L^\infty(\Omega)$,
with
\[
\begin{aligned}
\langle J_\lambda'(u),v\rangle
={}&
\int_\Omega\sum_{i,j=1}^N
a_{ij}(x,u)D_{x_i}uD_{x_j}v+\frac12\int_\Omega
\sum_{i,j=1}^N
D_s a_{ij}(x,u)D_{x_i}uD_{x_j}u\,v+\lambda\int_\Omega uv
-\int_\Omega g(u)v.
\end{aligned}
\]
The difficulty arises from the term
\[
\sum_{i,j=1}^N
D_s a_{ij}(x,u)D_{x_i}uD_{x_j}u,
\]
which generally belongs only to $L^1(\Omega)$
and need not belong to $(H^1(\Omega))'$.
Moreover, for every fixed
$v\in H^1(\Omega)\cap L^\infty(\Omega)$,
the quantity $\langle J_\lambda'(u),v\rangle$
depends continuously on $u$
(see \cite{CaninoMauroNeumann2026}).

Set
\[
L_\lambda=\inf\left\{J_\lambda(u)\ :\ \text{$u\not\equiv0$ is a weak solution of \eqref{P}}\right\}.
\]
A least energy solution of \eqref{P} is a minimizer of the level $L_\lambda$.

In the semilinear case, it is well known that least energy solutions can be obtained by minimizing the functional on the Nehari manifold
\[
\mathcal{N}_\lambda := \left\{ u \in H^1(\Omega)\setminus\{0\} \ :\ \langle J_\lambda'(u),u\rangle = 0 \right\}.
\]
Moreover, under suitable assumptions, this level coincides with the Mountain Pass level associated with $J_\lambda$ (see, e.g., \cite{willem1996minimax}).

In the present quasilinear setting, however, the use of the Nehari
manifold becomes more delicate. To see the difficulty, consider
for simplicity the pure-power nonlinearity
$g(s)=|s|^{p-2}s$. For every $u\in H^1(\Omega)\setminus\{0\}$,
the derivative of the fibering map $t\mapsto J_\lambda(tu)$
satisfies, for $t>0$,
\[
\begin{aligned}
\frac{1}{t}\frac{d}{dt}J_\lambda(tu)
={}&\int_\Omega\sum_{i,j=1}^N
\left[a_{ij}(x,tu)+\frac{tu}{2}D_s a_{ij}(x,tu)\right]
D_{x_i}uD_{x_j}u+\lambda\int_\Omega u^2
-t^{p-2}\int_\Omega|u|^p.
\end{aligned}
\]
Unlike the semilinear case, the contribution of the diffusion
coefficients now depends nontrivially on $t$.
Assumptions \eqref{a.1}--\eqref{a.3} do not ensure the
monotonicity required to obtain a unique positive zero of
this expression, and therefore do not guarantee a unique
critical point or a unique global maximum of the fibering map.

Consequently, the Nehari manifold need not be a natural
constraint, and the usual minimization argument cannot be
applied directly (see also \cite{szulkin2010method}).

{
Our approach is inspired by the variational method of Pankov \cite{pankov2005, szulkin2010method}, which uses a spectral decomposition and a generalized Nehari manifold to study least energy solutions in indefinite problems. In our quasilinear setting, however, the energy functional is not differentiable, and the generalized Nehari method cannot be applied directly. We therefore minimize the energy over the set of nontrivial weak solutions, whose existence follows from \cite{CaninoMauroNeumann2026}. Using compactness arguments, we show that the least energy level is attained. In the indefinite case, the spectral decomposition of the linearized operator and the non-resonance condition play an important role in preventing minimizing sequences from converging to zero.
}

\subsection{State of the art}

\begin{itemize}

\item \emph{The semilinear problem.}
Let $\lambda_1$ denote the first eigenvalue of the Laplacian with Dirichlet boundary conditions. If $-\lambda < \lambda_1$, it is classical that the problem
\begin{equation*}
-\Delta z+\lambda z=|z|^{p-2}z \quad \text{in }\Omega, 
\qquad z=0 \quad \text{on }\partial\Omega,
\end{equation*}
admits a weak solution. This can be proved via variational methods such as the Mountain Pass Theorem or the Nehari manifold approach (see, e.g., \cite[Sections 1.4 and 4.2]{willem1996minimax}). 

When $-\lambda \ge \lambda_1$, the existence of a nontrivial weak solution can still be established, for instance by means of a linking theorem \cite[Section 2.4]{willem1996minimax}, or alternatively through the generalized Nehari method \cite[Chapter 4]{szulkin2010method}. While the linking approach guarantees the existence of a critical point with suitable geometric properties, the generalized Nehari method allows one to characterize directly a solution of minimal energy among all nontrivial solutions. 

We now turn to the corresponding problem with Neumann boundary conditions:
\begin{equation*}
-\Delta z+\lambda z=|z|^{p-2}z \quad \text{in }\Omega, 
\qquad \frac{\partial z}{\partial\nu}=0 \quad \text{on }\partial\Omega.
\end{equation*}
This case exhibits significant differences. In particular, if $\lambda > 0$, constant solutions exist, namely $z \equiv \pm \lambda^{\frac{1}{p-2}}$. In the subcritical regime, the existence of nonconstant solutions for sufficiently large $\lambda > 0$ was established in \cite{lin1988large, lin2006diffusion}. 

On the other hand, if $\lambda \le 0$, constant solutions are no longer admissible, and one can adapt the techniques developed for the Dirichlet problem to obtain the existence of a ground state solution.

\item \emph{The quasilinear problem.}
Consider the Dirichlet problem
\begin{equation}
-\sum_{i,j=1}^N D_{x_j}\big(a_{ij}(x,u)D_{x_i}u\big)
+\frac12 \sum_{i,j=1}^N D_s a_{ij}(x,u) D_{x_i}u D_{x_j}u
= g(x,u)
\quad \text{in }\Omega, 
\qquad u=0 \quad \text{on }\partial\Omega.
\end{equation}

{
As already observed for the Neumann problem, the associated energy
functional is continuous on $H_0^1(\Omega)$,
but in general it is neither of class $C^1$
nor locally Lipschitz.
}

This class of problems has been extensively studied in 
\cite{caninoquasilineare1, caninoserdica, nonsmooththeory1}, 
where non-smooth critical point theory was employed to prove the existence of \emph{infinitely many bounded solutions}. 

In a related Dirichlet setting, Arcoya and Boccardo~\cite{arcoyaboccardo1} established the existence of a bounded solution by means of \emph{a priori} estimates.

Subsequently, Candela and Palmieri 
\cite{candela2006multiple, candela2009some, candela2009infinitely} 
obtained infinitely many bounded solutions for this equation and for a suitable generalization. Their approach relies on a differentiable critical point framework, achieved by working in the space 
\( W_0^{1,p}(\Omega) \cap L^\infty(\Omega) \), 
where the associated functional becomes smooth.

{
The symmetric multiplicity theorem in \cite{CaninoMauroNeumann2026} provides the existence result used in this paper; our purpose is to establish attainment and qualitative properties of the least energy level.
} However, the existence of a least energy solution for quasilinear problems has been addressed only in the Dirichlet case, see \cite[Section 4]{CaninoMauroCoopCompSystem2025}.

\end{itemize}
\subsection{Main results}
We begin by analyzing the case $\lambda>0$, for which we obtain the existence of least energy solutions together with their asymptotic behavior as $\lambda \to +\infty$.

\begin{theorem}\label{theorem lambda positive}
Assume that \eqref{a.1}-\eqref{a.3}, \eqref{g.1'}-\eqref{g.2} hold.
    The problem \eqref{P}
    has a least energy solution $u_\lambda\in H^1(\Omega)\cap L^\infty(\Omega)$ for every $\lambda>0$. If, in addition,
\[
\frac{{p}}{p-2}>\beta_q,
\qquad
\beta_q:=\frac{(N-2)(2^*-q)}{2(q-2)},
\]
then there exists $\lambda_0>0$ such that $u_\lambda$ is not constant for every $\lambda>\lambda_0$. Furthermore, $L_\lambda=J_\lambda(u_\lambda)\to+\infty$ as $\lambda\to+\infty$.
\end{theorem}
As a direct consequence, we deduce the corresponding result for the Lane--Emden type problem, where the nonlinearity is of pure power type.
{
Whenever we consider the power nonlinearity
$g(s)=|s|^{p-2}s$, the structural assumptions
\eqref{a.1}-\eqref{a.3} are understood with $q=p$;
in particular, $0<\gamma<p-2$.
}
\begin{corollary}\label{corollario lambda positivo}
    Assume that \eqref{a.1}-\eqref{a.3} hold.
    The problem:
    \begin{equation}\label{lane emden lambda problem}
\left\{
\begin{aligned}
-&\sum_{i,j=1}^ND_{x_j}(a_{ij}(x,u)D_{x_i}u)+\frac12\sum_{i,j=1}^ND_sa_{ij}(x,u)D_{x_i}uD_{x_j}u=|u|^{p-2}u-\lambda u\quad \text{in $\Omega$},\\
&\sum_{i,j=1}^Na_{ij}(x,u)D_{x_i}u\ \nu_j=0\quad\text{on $\partial\Omega$},
\end{aligned}
\right.
\end{equation}
    has a least energy solution $u_\lambda\in H^1(\Omega)\cap L^\infty(\Omega)$ for every $\lambda>0$, and there exists $\lambda_0>0$ such that $u_\lambda$ is not constant whenever $\lambda>\lambda_0$. Furthermore, $L_\lambda=J_\lambda(u_\lambda)\to+\infty$ as $\lambda\to+\infty$.
\end{corollary}

The divergence of the least energy level $L_\lambda\to+\infty$ as
$\lambda\to+\infty$ also raises the question of the asymptotic
behavior of least energy solutions.

Blow-up techniques have played a fundamental role
in the study of semilinear elliptic equations,
with classical contributions such as that of
Gidas and Spruck \cite{gidasspruck1981global},
and an extensive subsequent literature.

For Neumann problems, Ni and Takagi
\cite{ni1991shape,ni1993peaks} studied the concentration of least
energy solutions near the boundary, while related concentration
phenomena have been investigated in
\cite{pistoiaRamos2004,pellacciPistoiaVairaVerzini2021,
pistoiaSaldanaTavares2023,pistoiaSchiera2025}.

Moreover, in the power case, every nontrivial solution satisfies
$\|u_\lambda\|_{L^\infty(\Omega)}\ge
\lambda^{1/(p-2)}$ (see Remark~\ref{remark Linf growth}),
so its $L^\infty$-norm diverges as $\lambda\to+\infty$.
In the quasilinear setting, however, the analysis of possible
concentration phenomena is more delicate, since rescaling the
solutions also changes the diffusion coefficients.

{

We next consider the opposite regime $\lambda\to0^+$. In this case we prove that, under an additional regularity assumption on $g$, least energy solutions are necessarily constant for $\lambda>0$ sufficiently small.

\begin{proposition}\label{small lambda general}
Assume that \eqref{a.1}-\eqref{a.3}, \eqref{g.1'}-\eqref{g.2} hold and let $g\in C^1(\mathbb R)$. There exists $\widetilde\lambda_0>0$ such that, for every $0<\lambda<\widetilde\lambda_0$, every least energy solution of \eqref{P} is constant. In particular, if $u_\lambda\equiv c_\lambda\neq0$, then
\(
g(c_\lambda)=\lambda c_\lambda
\)
and
\[
L_\lambda=J_\lambda(c_\lambda)=|\Omega|\left(\frac{\lambda}{2}c_\lambda^2-G(c_\lambda)\right).
\]
\end{proposition}

For the Lane--Emden nonlinearity, the constant solutions can be explicitly determined, and Proposition~\ref{small lambda general} gives the following result.

\begin{corollary}\label{small lambda power}
Assume that \eqref{a.1}-\eqref{a.3} hold. There exists $\widetilde\lambda_0>0$ such that, for every $0<\lambda<\widetilde\lambda_0$, every least energy solution of \eqref{lane emden lambda problem} is constant. More precisely,
\(
u_\lambda\equiv\pm\lambda^{\frac1{p-2}},
\)
and
\(
L_\lambda=\frac{p-2}{2p}|\Omega|\lambda^{\frac{p}{p-2}}.
\)
\end{corollary}

The behavior described above is closely related to the
classical semilinear theory developed by Lin, Ni and Takagi
\cite{lin1988large,lin2006diffusion}.
For the pure-power semilinear equation, the change of
variables $u=\lambda^{1/(p-2)}v$ transforms
$-\Delta u+\lambda u=|u|^{p-2}u$ into
\(
-\frac{1}{\lambda}\Delta v+v=|v|^{p-2}v,
\)
with the same homogeneous Neumann boundary condition.

Thus, large values of $\lambda$ correspond to the
small-diffusion regime, while $\lambda\to0^+$ corresponds
to the large-diffusion regime.

In the quasilinear case, the same rescaling changes the
diffusion coefficients as well. For instance, when
$g(s)=|s|^{p-2}s$, the coefficients in the rescaled problem
become $a_{ij}(x,\lambda^{1/(p-2)}v)$, and hence depend
explicitly on $\lambda$; the semilinear rescaling argument
cannot therefore be applied unchanged.

Nevertheless,
Theorem~\ref{theorem lambda positive} and
Proposition~\ref{small lambda general} extend two qualitative
features of the semilinear theory to the quasilinear setting:
least energy solutions are nonconstant for sufficiently
large $\lambda$ (under the additional assumptions of the
theorem), whereas all least energy solutions are constant
for sufficiently small positive $\lambda$ (under the
regularity assumption of the proposition).
}

The proof of Proposition~\ref{small lambda general} is based on a smallness argument. First, using a constant solution $c_\lambda$ satisfying $g(c_\lambda)=\lambda c_\lambda$ as a competitor and condition \eqref{g.2}, we show that least energy solutions converge to zero in $H^1(\Omega)$ as $\lambda\to0^+$. A Moser iteration, uniform with respect to $\lambda$, then yields convergence to zero in $L^\infty(\Omega)$. Finally, once the $L^\infty$-smallness of least energy solutions has been established, we prove that every bounded solution with sufficiently small $L^\infty$-norm is constant.


{
\begin{remark}\label{remark Linf growth}
For $\lambda>0$, every nontrivial weak solution $u\in H^1(\Omega)\cap L^\infty(\Omega)$ of the problem \eqref{lane emden lambda problem} satisfies
\[
\|u\|_{L^\infty(\Omega)}\ge \lambda^{1/(p-2)}.
\]

Testing the equation with $u$ and using
\textup{(a.2)}--\textup{(a.3)}, we obtain
\[
\nu_0\int_\Omega |\nabla u|^2
+\lambda\int_\Omega u^2
\le \int_\Omega |u|^p
\le \|u\|_{L^\infty(\Omega)}^{p-2}\int_\Omega u^2.
\]
Since $u\not\equiv0$, division by $\int_\Omega u^2$
yields the claimed estimate.
\end{remark}
}

{
We then turn to the case $\lambda\le0$. In this regime,
nontrivial weak solutions are necessarily nonconstant.
We prove that the least energy level is attained in the
non-resonant case, namely when $0$ is not an eigenvalue
of the linearized Neumann operator
\[
\mathcal L_\lambda u
=
-\sum_{i,j=1}^N
D_{x_j}\bigl(a_{ij}(x,0)D_{x_i}u\bigr)+\lambda u.
\] }

\begin{theorem}\label{theorem lambda<0 non resonant case}
   Assume that \eqref{a.1}-\eqref{a.3}, \eqref{g.1'}-\eqref{g.2} and \eqref{symmetry} hold.  For every $\lambda\le0$ the problem \eqref{P}
has infinitely many distinct pairs of nonconstant weak
solutions in $H^1(\Omega)\cap L^\infty(\Omega)$. 
    
   Furthermore, if 
    \begin{equation}\label{non-resonant hypothesis}
        \text{$0$ is not an eigenvalue of}\quad \mathcal L_{\lambda} (u)=-\sum_{i,j=1}^ND_{x_j}(a_{ij}(x,0)D_{x_i} u)+\lambda u,
    \end{equation}

   then there exists a nonconstant least energy solution of \eqref{P}.

\end{theorem}

{
To treat the resonance at $\lambda=0$, we restrict ourselves to the power nonlinearity and consider the corresponding Lane--Emden problem.
} In this setting, we recover the existence of least energy solutions and describe the behavior of the least energy level as $\lambda \to 0^+$ for the perturbed problem \eqref{lane emden lambda problem}.

\begin{theorem}\label{lane emden theorem}
    Assume that \eqref{a.1}-\eqref{a.3}  hold. The problem:
    \begin{equation}\label{lane emden problem}
    \left\{
    \begin{aligned}
-&\sum_{i,j=1}^ND_{x_j}(a_{ij}(x,u)D_{x_i}u)+\frac12\sum_{i,j=1}^ND_sa_{ij}(x,u)D_{x_i}uD_{x_j}u=|u|^{p-2}u\quad \text{in $\Omega$},\\
&\sum_{i,j=1}^Na_{ij}(x,u)D_{x_i}u\ \nu_j=0\quad\text{on $\partial\Omega$},
\end{aligned}
\right.
\end{equation}
has a least energy solution $u_0\in H^1(\Omega)\cap L^\infty(\Omega)$, which is not constant. 

Furthermore, if $L_\lambda$ denotes the least energy level associated with \eqref{lane emden lambda problem} for $\lambda>0$ and with \eqref{lane emden problem} for $\lambda=0$, then $L_\lambda\to0$ as $\lambda\to0^+$, while $L_0>0$. In particular, $\lambda\mapsto L_\lambda$ is not continuous at $\lambda=0$. 

\end{theorem}


We now consider the semilinear case
$a_{ij}(x,s)=a_{ij}(x)$, with the power
nonlinearity $g(s)=|s|^{p-2}s$.
In this setting, the preceding results provide
a more complete description of least energy
solutions, including resonant values of $\lambda$.
In particular, for $\lambda>0$, the nonzero
constant solutions are precisely
$u\equiv\pm\lambda^{1/(p-2)}$.
For sufficiently small positive $\lambda$,
these solutions attain the least energy level,
whereas for sufficiently large $\lambda$
every least energy solution is nonconstant.

The following corollary summarizes these results
and describes the behavior of the least energy
level for all $\lambda\in\mathbb R$.

\begin{corollary}\label{semilinear case}
Assume that \eqref{a.1}-\eqref{a.2} hold and that $a_{ij}(x,s)=a_{ij}(x)$ for every $i,j=1,\dots,N$. The problem
\begin{equation}\label{semilinear problem}
\left\{
\begin{aligned}
-&\sum_{i,j=1}^N D_{x_j}(a_{ij}(x)D_{x_i}u)+\lambda u=|u|^{p-2}u &&\text{in $\Omega$},\\
&\sum_{i,j=1}^N a_{ij}(x)D_{x_i}u\,\nu_j=0 &&\text{on $\partial\Omega$},
\end{aligned}
\right.
\end{equation}
has a least energy solution $u_\lambda\in H^1(\Omega)\cap L^\infty(\Omega)$ for every $\lambda\in\R$. Moreover,

\begin{itemize}
\item There exists $\widetilde\lambda_0>0$ such that, for
      every $0<\lambda<\widetilde\lambda_0$, all least
      energy solutions are constant, namely
      $u_\lambda\equiv\pm\lambda^{1/(p-2)}$, and
      \(
      L_\lambda=\frac{p-2}{2p}|\Omega|\lambda^{p/(p-2)}.
      \)
\item As $\lambda\to0^+$, one has $L_\lambda\to0$,
      whereas $L_0>0$ is attained by a nonconstant least
      energy solution. In particular, $\lambda\mapsto
      L_\lambda$ is discontinuous at $\lambda=0$.
\item As $\lambda\to+\infty$, one has
      $L_\lambda\to+\infty$. Moreover, there exists
      $\lambda_0>0$ (as in
      Corollary~\ref{corollario lambda positivo}) such that
      every least energy solution is nonconstant whenever
      $\lambda>\lambda_0$.
\item For every $\lambda<0$, the level $L_\lambda>0$
      is attained by a nonconstant least energy solution.
\end{itemize}

\end{corollary}
\begin{center}
\begin{tikzpicture}
\begin{axis}[
    width=12cm,
    height=7cm,
    axis lines=middle,
    xlabel={$\lambda$},
    ylabel={$L_\lambda$},
    xmin=-0.05, xmax=1.1,
    ymin=-0.05, ymax=1.0,
    xtick={0,1},
    xticklabels={$0$,$\widetilde\lambda_0$},
    ytick=\empty,
    samples=200,
    domain=0.001:1,
    clip=false,
]

\pgfmathsetmacro{\p}{4}
\pgfmathsetmacro{\Lzero}{0.8}

\addplot[very thick, blue] {((\p-2)/(2*\p)) * x^(\p/(\p-2))};

\addplot[mark=*, blue] coordinates {(0,\Lzero)};
\node[left] at (axis cs:0,\Lzero) {$L_0>0$};

\addplot[mark=o, blue, fill=white] coordinates {(0,0)};
\node[below right] at (axis cs:0,0) {$\lim\limits_{\lambda\to0^+}L_\lambda=0$};

\node[above right] at (axis cs:0.45,0.08)
{{$\displaystyle L_\lambda=\frac{p-2}{2p}|\Omega|\lambda^{\frac{p}{p-2}}$}};

\end{axis}
\end{tikzpicture}
\end{center}

\begin{remark}
In the semilinear case $a_{ij}(x)=\delta_{ij}$, the existence of least energy solutions was already established in \cite{szulkin2009ground} by means of the generalized Nehari manifold $($see also \cite{szulkin2010method}$)$.

\end{remark}

\begin{remark}
 If the
diffusion matrix $A=(a_{ij})$ is constant, the affine change
\[
y=A^{-1/2}x,
\qquad
\widetilde\Omega=A^{-1/2}\Omega,
\]
transforms
\[
-\operatorname{div}(A\nabla u)+\lambda u=u^{p-1}
\quad\text{in }\Omega
\]
into the classical Ni--Takagi problem
\[
-\Delta v+\lambda v=v^{p-1}
\quad\text{in }\widetilde\Omega.
\]
Define
\[
H_A(x):=
H_{\partial\widetilde\Omega}(A^{-1/2}x),
\qquad x\in\partial\Omega.
\]
{
Assume that $\partial\Omega$ is smooth. Let $P_\lambda$ be a
global maximum point of a positive least energy solution, and
let $Q_\lambda\in\partial\Omega$ be a nearest boundary point
to $P_\lambda$. The classical peak-location result of
\cite{ni1993peaks} yields
\[
\lim_{\lambda\to+\infty}H_A(Q_\lambda)
=\max_{\partial\Omega}H_A.
\]
Thus, in the constant-coefficient case, the concentration points, after the change of coordinates, approach points of $\partial\widetilde\Omega$ where the mean curvature is maximal.
}
\end{remark}

\medskip
\noindent\textbf{Structure of the paper.}
{
Section~2 recalls the nonsmooth critical point framework and the auxiliary compactness tools used throughout the paper. 
 
{

Section~3 treats the definite case $\lambda>0$, proving the existence of least energy solutions and studying their behavior as $\lambda\to+\infty$. Under the additional assumption $g\in C^1(\mathbb R)$, we also prove the nonexistence of nonconstant least energy solutions for sufficiently small $\lambda>0$.
}

Section~4 addresses the indefinite and non-resonant case $\lambda\le0$, establishing the existence of nontrivial solutions and, under the non-resonance condition, of least energy solutions. Section~5 treats the resonant case for the Lane--Emden nonlinearity, proving the existence of a nonconstant least energy solution at $\lambda=0$ and describing the discontinuity of the least energy level as $\lambda\to0^+$.
}

\section{Preliminaries}
We recall some results about the critical point theory of continuous functionals, developed in \cite{nonsmooththeory1, corvellec1993deformation,degiovanni1994critical}. In this setting, we consider $(X,d)$ a metric space and $f: X \to \R$ a continuous functional.

\begin{definition} \label{definition 1}
Let $(X,d)$ be a metric space and let $f: X \to \mathbb{R}$ be a continuous function. We consider $\sigma \ge 0$ such that there exist $\delta > 0$ and a continuous map $\mathscr{H}: B_{\delta}(u) \times [0, \delta] \to X$ such that
\begin{align*}
d(\mathscr{H}(v, t), v) \le t, \qquad f(\mathscr{H}(v, t)) \le f(v) - \sigma t.
\end{align*}
  We denote by $|df|(u)$ the supremum of all such $\sigma$ and we call it the \emph{weak slope} of $f$ at $u$.
\end{definition}


\begin{definition}\label{def p.critical}
Let $(X,d)$ be a metric space and let $f: X \to \mathbb{R}$ be continuous. We will say that $u \in X$ is a (lower) critical point if $|df|(u) = 0$. A (lower) critical point is said to be at the level $c \in \mathbb{R}$ if it is also true that $f(u) = c$.
\end{definition}

\begin{definition}\label{ps}
Let $(X,d)$ be a metric space and let $f:X\to\R$ be a continuous functional. A sequence $\{u_n\}\subset X$ is called a $(PS)_c$-sequence if
\begin{align*}
&f(u_n)\to c,\qquad |df|(u_n)\to0.
\end{align*}
Furthermore, we will say that $f$ satisfies the $(PS)_c$-condition if every subsequence admits a convergent subsequence in $X$. 
If the $(PS)_c$-condition holds for every $c\in\R$ we will simply write $(PS)$-condition.
\end{definition}
We define: $$f^c:=\{u\in X\ :\ f(u)\le c\},$$
$$K=\{u\in X\ :\ |df|(u)=0\},\qquad K_c=\{u\in X\ :\ f(u)=c,\ |df|(u)=0\}.$$
Notice that $f^c$ is closed and $K_c$ is compact if $f$ satisfies the $(PS)_c$-condition.

\begin{theorem}[Deformation Lemma, {\cite[Theorem 1.1.13]{nonsmooththeory1}}]\label{defloc}
Let $(X,d)$ be a complete metric space, let $f:X\to\R$ be a continuous functional, and suppose that $f$ satisfies the $(PS)_c$-condition. Then, given $\bar\varepsilon>0$, a neighborhood $U$ of $K_c$ $($$U=\emptyset$ if $K_c=\emptyset$$)$ and $\lambda>0$, there exist $\varepsilon>0$ and $\eta:X\times[0,1]\to X$ continuous such that for every $u\in X$ and $t\in[0,1]$ we have that 
\begin{itemize}
\item[$(i)$] $d(\eta(u,t),u)\le\lambda t$,
\item[$(ii)$] $f(\eta(u,t))\le f(u)$,
\item[$(iii)$] $\eta(u,t)=u$ if $u\notin f^{-1}(c-\bar\varepsilon,c+\bar\varepsilon)$,
\item[$(iv)$] $\eta(f^{c+\varepsilon}\setminus U,1)\subseteq f^{c-\varepsilon}$.
\end{itemize}
\end{theorem}

We recall some relevant properties of the energy functional 
$J_\lambda:H^1(\Omega)\to\mathbb{R}$ established in \cite{CaninoMauroNeumann2026}.

\begin{proposition}[{\cite[Proposition 2.6]{CaninoMauroNeumann2026}}]
Assume that \eqref{a.1} and \eqref{g.1'} hold.
The functional $ J_\lambda$ is continuous in $H^1(\Omega)$.
Furthermore,
\begin{equation}\label{dis weak slope}
  |d J_\lambda|( u)\ge\sup_{\substack{ v\in C^\infty(\overline\Omega)\\\| v\|_{H^1}\le1}}\langle  J_\lambda'( u), v\rangle,
\end{equation}
where
\begin{align}
    \langle J_\lambda'(u),v\rangle&=\int_\Omega\sum_{i,j=1}^Na_{ij}(x,u)D_{x_i}uD_{x_j}v+\frac12\int_\Omega\sum_{i,j=1}^N\left(D_sa_{ij}(x,u)D_{x_i}uD_{x_j}u\right)v-\int_\Omega g(u)v+\lambda\int_\Omega uv
\end{align}
\end{proposition}

\begin{definition}
    Let $\{ u_n\}\subset H^1(\Omega)$ be a sequence. We say that $\{u_n\}$ is a Concrete Palais-Smale sequence at level $c\in\R$ if $J_\lambda(u_n)\to c$ and
    \begin{align}
  -\sum_{i,j=1}^ND_{x_j}(a_{ij}(x,u_n)D_{x_i}u_n)&+\frac12\sum_{i,j=1}^ND_sa_{ij}(x,u_n)D_{x_i}u_n D_{x_j}u_n-g(u_n)+\lambda u_n\to0\quad \text{in $\left(H^1(\Omega)\right)'$.}
    \end{align}
    The functional $J_\lambda$ satisfies the Concrete Palais-Smale condition at level $c\in\R$ if every Concrete Palais-Smale sequence at level $c\in\R$ admits a convergent subsequence in $H^1(\Omega)$.
\end{definition}

\begin{proposition}[{\cite[Proposition 2.8]{CaninoMauroNeumann2026}}]\label{PS e CPS}
The following facts hold:
\begin{enumerate}
\item[$(i)$] if $u$ is a (lower) critical point for $J_\lambda$, then $u$ is a weak solution of the problem \eqref{P},
\item[$(ii)$] each $(PS)_c$-sequence is also a $(CPS)_c$-sequence, for every $c\in\R$,
\item[$(iii)$] $J_\lambda$ satisfies $(CPS)_c$-condition$\implies J_\lambda$ satisfies $(PS)_c$-condition.
\end{enumerate}
\end{proposition}

\begin{theorem}\label{convergenza e soluzioni deboli}
    Let $\{u_n\}\subset H^1(\Omega)$ be a sequence such that $u_n\wto u$ in $H^1(\Omega)$, $\omega_n\to \omega$ in $\left(H^1(\Omega)\right)'$, and
    \[
    \int_\Omega\sum_{i,j=1}^Na_{ij}(x,u_n)D_{x_i}u_n D_{x_j}\varphi+\frac12\int_\Omega\sum_{i,j=1}^N\left(D_sa_{ij}(x,u_n) D_{x_i}u_n D_{x_j}u_n\right)\varphi=\langle\omega_n,\varphi\rangle,
    \]
    for every $\varphi\in H^1(\Omega)\cap L^\infty(\Omega)$. Then 
    \[
     \int_\Omega\sum_{i,j=1}^Na_{ij}(x,u)D_{x_i}u D_{x_j}\varphi+\frac12\int_\Omega\sum_{i,j=1}^N\left(D_sa_{ij}(x,u) D_{x_i}u D_{x_j}u\right)\varphi=\langle\omega,\varphi\rangle.
    \]
\end{theorem}
\begin{proof}
   It follows from \cite[Theorem 4.1, Step 1, equation (4.5)]{CaninoMauroNeumann2026} that the statement holds for every $\varphi \in C^\infty(\overline{\Omega})$.  Since $C^\infty(\overline{\Omega})$ is dense in $H^1(\Omega)\cap L^\infty(\Omega)$, the conclusion follows by a standard density argument.
\end{proof}

\begin{theorem}[{\cite[Theorem 3.1]{CaninoMauroNeumann2026}}]\label{brezis browder}
    Assume that \eqref{a.1}-\eqref{a.3} hold.
    Let $u\in H^1(\Omega)$ and $ \omega\in \left(H^1(\Omega)\right)'$ such that
    \begin{equation}\label{eq thm density}
   \int_\Omega \sum_{i,j=1}^Na_{ij}(x,u)D_{x_i}u D_{x_j} v+\frac12\int_\Omega \sum_{i,j=1}^N\left(D_sa_{ij}(x,u)D_{x_i}uD_{x_j}u\right)v= \langle\omega,v\rangle,
    \end{equation}
    for every $v\in C^\infty(\overline\Omega)$.
    Then 
    \[
    \int_\Omega\sum_{i,j=1}^Na_{ij}(x,u)D_{x_i}u D_{x_j}u+\frac12\int_\Omega \sum_{i,j=1}^N\left(D_sa_{ij}(x,u)D_{x_i}u D_{x_j}u\right)u=\langle\omega, u\rangle.
    \]
\end{theorem}

\begin{theorem}\label{PS condition}
    Assume that \eqref{a.1}-\eqref{a.3} and \eqref{g.1'}-\eqref{g.2} hold. The energy functional $J_\lambda$ satisfies the $(CPS)_c$-condition and the $(PS)_c$-condition for every $c,\lambda\in\R$.
\end{theorem}
\begin{proof}
   The functional $J$ satisfies the $(CPS)$-condition by \cite[Theorem 4.4 and Corollary 4.2]{CaninoMauroNeumann2026}. Since every $(PS)$-sequence is also a $(CPS)$-sequence, by Proposition \ref{PS e CPS}, the $(PS)$-condition follows.
\end{proof}
\begin{theorem}[{\cite[Theorem 3.4]{CaninoMauroNeumann2026}}]\label{regolarità}
  Assume that \eqref{a.1}-\eqref{a.3} and \eqref{g.1'} hold.
    Let $u\in H^1(\Omega)$ be a weak solution of 
    \[
    \left\{
\begin{aligned}
-&\sum_{i,j=1}^ND_{x_j}(a_{ij}(x,u)D_{x_i}u)+\frac12\sum_{i,j=1}^ND_sa_{ij}(x,u)D_{x_i}uD_{x_j}u=g(u)-\lambda u,\quad \text{in $\Omega$},\\
&\sum_{i,j=1}^Na_{ij}(x,u)D_{x_i}u\ \nu_j=0,\quad\text{on $\partial\Omega$}.
\end{aligned}
\right.
    \]
    Then $u\in L^\infty(\Omega)$.
\end{theorem}

{

\begin{theorem}[\cite{CaninoMauroNeumann2026}]\label{known existence}
Assume \eqref{a.1}--\eqref{a.3}, \eqref{g.1'}--\eqref{g.2}, and \eqref{symmetry}. For every $\lambda\in\mathbb R$, there exists a sequence of nontrivial weak solutions $u_j\in H^1(\Omega)\cap L^\infty(\Omega)$ of \eqref{P} such that
\[
J_\lambda(u_j)\longrightarrow+\infty.
\]
In particular, the set of nontrivial weak solutions is nonempty. If $\lambda\le0$, every nontrivial weak solution is nonconstant, and there are infinitely many distinct solution pairs $\{u,-u\}$.
\end{theorem}

}

{
\begin{lemma}\label{least energy level finite}
Assume that \eqref{a.1}--\eqref{a.3},
\eqref{g.1'}--\eqref{g.2}, and \eqref{symmetry} hold. Then
\[
-\infty<L_\lambda<+\infty
\qquad\text{for every }\lambda\in\mathbb R.
\]
\end{lemma}

\begin{proof}
Fix $\lambda\in\mathbb R$ and let $u$ be a nontrivial weak solution
of \eqref{P}. Testing the equation with $u$, as allowed by
Theorem~\ref{brezis browder}, and using \eqref{a.3} and \eqref{g.2},
we obtain
\begin{align*}
(2+\gamma)J_\lambda(u)
={}&\frac12\int_\Omega\sum_{i,j=1}^{N}
\bigl[\gamma a_{ij}(x,u)-uD_sa_{ij}(x,u)\bigr]
D_{x_i}uD_{x_j}u\\
&+\frac{\gamma\lambda}{2}\int_\Omega u^2
+\int_\Omega\bigl[g(u)u-(2+\gamma)G(u)\bigr]\\
\ge{}&
\frac{\gamma\lambda}{2}\int_\Omega u^2
+(q-2-\gamma)\int_\Omega G(u).
\end{align*}
Moreover, condition \eqref{g.2} yields constants $a,b>0$ such that
\begin{equation}\label{stima inferiore G}
G(s)\ge a|s|^q-b
\qquad\text{for every }s\in\mathbb R.
\end{equation}
Since $q>2$ and $\gamma<q-2$, the function
\begin{equation}
s\longmapsto
(q-2-\gamma)G(s)+\frac{\gamma\lambda}{2}s^2
\end{equation}
is bounded below on $\mathbb R$. Consequently, there exists
$C_\lambda>0$, independent of $u$, such that
\[
J_\lambda(u)\ge-C_\lambda.
\]
Thus $L_\lambda>-\infty$. Finally, Theorem~\ref{known existence}
guarantees the existence of a nontrivial weak solution, and hence
$L_\lambda<+\infty$.
\end{proof}
}


\section{Definite case}
\begin{lemma}\label{geometry}
Assume that \eqref{a.1}-\eqref{a.2} and \eqref{g.1'}-\eqref{g.2} hold. 
Let $\lambda>0$.  
  There exist $\alpha,\rho>0$ such that
    \begin{itemize}
        \item[$(i)$] $J_\lambda(u)\ge\alpha$ for every $u\in\partial B_\rho$;
        \item[$(ii)$] for every $v_0\in H^1(\Omega)\setminus\{0\}$, there exists $t_0>0$ such that $J_\lambda(t_0v_0)<0$.
    \end{itemize}
\end{lemma}
\begin{proof}
    $(i)$ Let $\varepsilon>0$. 
    {
    
    By the mean value theorem, for every $s\in\mathbb{R}$ there exists $\theta\in(0,1)$ such that
\(
G(s)=g(\theta s)s.
\)
Hence, using \eqref{g.1}, we obtain
\[
|G(s)|=|G(s)-G(0)|=|g(\theta s)s|\leq \varepsilon s^2+C_\varepsilon |s|^p.
\]}
    By \eqref{a.2}, we obtain:
    \begin{align}
        J_\lambda(u)&\ge\frac{\nu_0}{2}\int_\Omega|\nabla u|^2+\left(\frac{\lambda}{2}-\varepsilon\right)\int_\Omega u^2-C_\varepsilon\int_\Omega |u|^p\\
        &\ge\frac{\nu_0}{2}\left[\int_\Omega|\nabla u|^2+ \frac{\lambda-2\varepsilon}{\nu_0}\int_\Omega u^2\right]-C_\varepsilon\int_\Omega|u|^p\\
        &\ge c_\varepsilon\|u\|^2-C_\varepsilon\int_\Omega|u|^p,
    \end{align}
    where $c_\varepsilon>0$ for $\varepsilon>0$ small enough. Thus, for every $u\in\partial B_\rho$, we have
    \[
    J_\lambda(u)\ge c_\varepsilon \rho^2-C_\varepsilon'\rho^p,
    \]
    where $C_\varepsilon'>0$. The claim follows by taking $\rho>0$ small enough.

    $(ii)$ Let $v_0\in H^1(\Omega)$ with $v_0\not\equiv0$ and consider
    \[
    J_\lambda(tv_0)=\frac{t^2}{2}\int_\Omega\sum_{i,j=1}^Na_{ij}(x,tv_0)D_{x_i}v_0D_{x_j}v_0+\frac{\lambda t^2}{2}\int_\Omega v_0^2-\int_\Omega G(tv_0),
    \]
    for every $t>0$.
    By \eqref{stima inferiore G}, we have
    \begin{equation}\label{stima G}
    \int_\Omega G(tv_0)\ge a_1t^q\int_\Omega|v_0|^q-a_2,
    \end{equation}
    and, from \eqref{a.1}, we get
    \[
    J_\lambda(tv_0)\le\frac{NC_0}{2}t^2\|v_0\|^2+\frac{\lambda}{2}t^2|v_0|_2^2-a_1t^q\int_\Omega |v_0|^q+a_2\to-\infty,
    \]
    as $t\to+\infty$. In particular, there exists $t_0>0$ such that $J_\lambda(t_0v_0)<0$. 
\end{proof}
Consider the Mountain Pass level:
\[
c_\lambda^{MP}=\inf_{h\in\Gamma}\max_{t\in[0,1]}J_\lambda(h(t)),
    \]
    where $\Gamma=\left\{h:[0,1]\to H^1(\Omega)\ \text{continuous}\ :\ h(0)=0,\ J_\lambda(h(1))<0\right\}$.
    According to Lemma \ref{geometry}, it follows that $\alpha\le c_\lambda^{MP}<+\infty$ and $\Gamma\ne\emptyset$.

\begin{proposition}\label{clambda critico}
Assume that \eqref{a.1}-\eqref{a.3}, \eqref{g.1'}-\eqref{g.2} hold.
    The level $c_\lambda^{MP}$ is critical.
\end{proposition}
\begin{proof}
{
The following result is a consequence of \cite[Theorem 1.3.1]{nonsmooththeory1}. We include the proof for completeness.
}
    Assume by contradiction that $c_\lambda^{MP}$ is not a critical level and let $\delta>0$ be such that
    \[
    \max_{t\in[0,1]}J_\lambda(h(t))\le c_\lambda^{MP}+\delta,
    \]
for some $h\in\Gamma$. Since $J_\lambda$ satisfies the $(PS)$-condition by Theorem \ref{PS condition}, Theorem \ref{defloc} implies that there exists $\eta:H^1(\Omega)\times[0,1]\to H^1(\Omega)$ continuous such that $J_\lambda(\eta(h([0,1]),1))\subseteq (-\infty,c_\lambda^{MP}-\delta)$ and $\eta(h(s),1)\in\Gamma$. Thus,
\[
c_\lambda^{MP}\le\max_{s\in[0,1]}J_\lambda(\eta(h(s),1))\le c_\lambda^{MP}-\delta<c_\lambda^{MP},
\]
which is absurd.
\end{proof}
Define the least energy level
\[
L_\lambda=\inf\left\{J_\lambda(u)\ :\ \text{$u\not\equiv0$ is a weak solution of \eqref{P}}\right \}.
\]
From Proposition \ref{clambda critico}, we have  $L_{\lambda}\le c_\lambda^{MP}$.

\begin{proposition}\label{stima energia}
  There exist $C_1>0$ and {$\Lambda>0$} such that
\[
L_\lambda\le C_1\lambda^{\beta_q}
\qquad {\text{for every }\lambda\ge\Lambda}\ \text{ with}\ \beta_q:=\frac{(N-2)(2^*-q)}{2(q-2)}.
\]
    
\end{proposition}
\begin{proof}
    Let $v_0\in H^1(\Omega)\setminus\{0\}$ and $v=t_0v_0$ as in Lemma \ref{geometry}. Then the map
    \[
    [0,1]\ni t\mapsto tv\in H^1(\Omega)
    \]
    belongs to $\Gamma$ and 
    \[
    c_\lambda^{MP}\le\inf_{v_0\in H^1(\Omega)\setminus\{0\}}\max_{t>0}J_\lambda(tv_0).
    \]
    Without loss of generality, up to a translation of the domain, we may assume that $0\in\Omega$; hence there exists $r_0>0$ such that $B_{r_0}(0)\subset\Omega$, and for $\lambda>r_0^{-2}$ the support of $w_\lambda$ is contained in $\Omega$.
    Let $w_\lambda\in H^1(\Omega)\cap L^\infty(\Omega)$ be the function
    \[
    w_\lambda=
    \begin{cases}
        \lambda^{\frac N2}(1-\sqrt\lambda|x|)&\sqrt\lambda|x|\le1\\
        0&\sqrt\lambda|x|\ge1.
    \end{cases}
    \]

    By \cite[eq. $(2.11)-(2.12)$]{lin1988large}, one has that for each $r\ge1$
    \[
    \int_\Omega|w_\lambda|^r=\kappa_r\lambda^{(r-1)\frac N2},\qquad \int_\Omega|\nabla w_\lambda|^2=\kappa\lambda^{1+\frac N2},
    \]
    where $\kappa,\kappa_r>0$ are independent of $\lambda$. Thus, taking into account \eqref{a.1} and \eqref{stima G},
    \begin{equation}\label{stima superiore L_lambda}
    \begin{aligned}
        J_\lambda(tw_\lambda)&=\frac{t^2}{2}\int_\Omega\sum_{i,j=1}^Na_{ij}(x,tw_\lambda)D_{x_i} w_\lambda D_{x_j}w_\lambda+\frac\lambda 2 t^2\int_\Omega w_\lambda^2-\int_\Omega G(tw_\lambda)\\
        &\le\frac{NC_0}{2}t^2\int_\Omega|\nabla w_\lambda|^2+\frac{\lambda}{2}t^2\int_\Omega w_\lambda^2-a_1t^q\int_\Omega |w_\lambda|^q+a_2.
    \end{aligned}
    \end{equation}

{
Set
\[
A_\lambda:=\frac{NC_0}{2}\int_\Omega|\nabla w_\lambda|^2
+\frac{\lambda}{2}\int_\Omega w_\lambda^2,
\qquad
B_\lambda:=a_1\int_\Omega|w_\lambda|^q.
\]
The function $t\mapsto A_\lambda t^2-B_\lambda t^q$ attains its
maximum at
\[
t_*=\left(\frac{2A_\lambda}{qB_\lambda}\right)^{\frac{1}{q-2}}.
\]
Consequently,
\begin{align*}
\max_{t>0}J_\lambda(tw_\lambda)
&\le\max_{t>0}\left(A_\lambda t^2-B_\lambda t^q\right)+a_2=\frac{q-2}{q}\left(\frac{2}{q}\right)^{\frac{2}{q-2}}
\frac{A_\lambda^{\frac{q}{q-2}}}{B_\lambda^{\frac{2}{q-2}}}+a_2=C\lambda^{\beta_q}+a_2,
\end{align*}
where $C>0$ is independent of $\lambda$.
Thus, for all sufficiently large $\lambda$,
\[
L_\lambda\le c_\lambda^{MP}
\le\max_{t>0}J_\lambda(tw_\lambda)
\le (C+a_2)\lambda^{\beta_q}.
\]
}
\end{proof}

\begin{proof}[Proof of Theorem \ref{theorem lambda positive}]
\textbf{Step 1}: Existence of a least energy solution.

    {
By Lemma \ref{least energy level finite} , $L_\lambda$ is well defined. Hence there exists a sequence
} $\{u_n\}\subset H^1(\Omega)$ such that 
    \[
    J_\lambda(u_n)\to L_\lambda,\quad \text{$u_n\not\equiv0$ solves \eqref{P} for every $n\ge1$.}
    \]
    {

By Theorem~\ref{PS condition}, up to a subsequence,
$u_n\to u$ strongly in $H^1(\Omega)$, where $u$
is a weak solution satisfying $J_\lambda(u)=L_\lambda$.
It remains to show that $u\neq0$.

}

    According to \eqref{a.3} and by Theorem \ref{brezis browder}, we can test \eqref{P} with $u_n$:
    \begin{align*}
    \int_\Omega\sum_{i,j=1}^Na_{ij}(x,u_n)D_{x_i}u_nD_{x_j}u_n
    &=\int_\Omega g(u_n)u_n-\lambda\int_\Omega u_n^2-\frac12\int_\Omega\sum_{i,j=1}^N
D_s a_{ij}(x,u_n)D_{x_i}u_nD_{x_j}u_n\,u_n\\
&\le\int_\Omega g(u_n)u_n-\lambda\int_\Omega u_n^2.
    \end{align*}
    From \eqref{g.1}, we have
    \[
    \int_\Omega g(u_n)u_n\le \frac{\lambda}{2}\int_\Omega u_n^2+C_\lambda\int_\Omega |u_n|^p.
    \]
    Finally, \eqref{a.2} implies that
    \[
    \nu_\lambda\|u_n\|^2\le\nu_0\int_\Omega|\nabla u_n|^2+\frac{\lambda}{2}\int_\Omega u_n^2\le C_\lambda'\|u_n\|^p,
    \]
    for some $\nu_\lambda,C_\lambda'>0$, so $\|u_n\|^{p-2}\ge\nu_\lambda/C_\lambda'$. Hence, $u\not\equiv0$ is a least energy solution. Moreover, Theorem \ref{regolarità} implies that $u\in H^1(\Omega)\cap L^\infty(\Omega)$.

   {
\medskip
\noindent
\textbf{Step 2}: Non-constancy for large $\lambda$ under the
additional assumption.

Assume that
\[
\frac{p}{p-2}>\beta_q.
\]
Let $k\ne0$ be a constant solution of $(\mathcal P_\lambda)$.
Then $g(k)=\lambda k$. By (1.2), with $\varepsilon=1$,
there exists $C>0$, independent of $\lambda$ and $k$, such that
\[
\lambda=\frac{g(k)}{k}\le 1+C|k|^{p-2}.
\]
Consequently, for $\lambda\ge2$,
\[
|k|^2\ge
\left(\frac{\lambda-1}{C}\right)^{2/(p-2)}
\ge c\lambda^{2/(p-2)}
\]
for some $c>0$ independent of $\lambda$ and $k$.
Moreover, by \textup{(g.2)},
\[
\begin{aligned}
J_\lambda(k)=|\Omega|\left(\frac{\lambda}{2}k^2-G(k)\right)\ge |\Omega|\left(\frac12-\frac1q\right)\lambda k^2\ge c_0\lambda^{p/(p-2)},
\end{aligned}
\]
where $c_0>0$ is independent of $\lambda$ and $k$.

On the other hand, Proposition~3.3 yields
\[
L_\lambda\le C_1\lambda^{\beta_q}
\qquad\text{for every }\lambda\ge\Lambda.
\]
{
Since $p/(p-2)>\beta_q$, there exists
$\lambda_0\ge\max\{2,\Lambda\}$ such that, for every
$\lambda>\lambda_0$ and every nonzero constant solution $k$,
\[
L_\lambda\le C_1\lambda^{\beta_q}
<c_0\lambda^{p/(p-2)}
\le J_\lambda(k).
\]
}
Thus, every least energy solution
is nonconstant for $\lambda>\lambda_0$.
}

\textbf{Step 3}: $L_\lambda\to+\infty$ as $\lambda\to+\infty$.

 {
 Let $u_\lambda\in H^1(\Omega)\cap L^\infty(\Omega)$ be a least energy solution of \eqref{P}.
 
Assume by contradiction that there exists a constant $ M > 0$, independent of $\lambda$, such that for arbitrarily large $\lambda$ one has
\begin{equation}\label{ipotesi contraddizione Theorem 1.1}
L_\lambda = J_\lambda(u_\lambda) \le M.
\end{equation} From \eqref{g.2} and by testing \eqref{P} with $u_\lambda$, we deduce that
 \begin{align}
     L_\lambda=J_\lambda(u_\lambda)&=\frac12\int_\Omega\sum_{i,j=1}^Na_{ij}(x,u_\lambda)D_{x_i}u_\lambda D_{x_j}u_\lambda+\frac{\lambda}{2}\int_\Omega u_\lambda^2-\int_\Omega G(u_\lambda)\\
     &\ge\frac{1}{2}\int_\Omega\sum_{i,j=1}^Na_{ij}(x,u_\lambda)D_{x_i}u_\lambda D_{x_j}u_\lambda+\frac\lambda2\int_\Omega u_\lambda^2-\frac1q\int_\Omega g(u_\lambda)u_\lambda\\
     &=\frac{q-2}{2q}\int_\Omega\sum_{i,j=1}^Na_{ij}(x,u_\lambda)D_{x_i}u_\lambda D_{x_j}u_\lambda+\frac{(q-2)\lambda}{2q}\int_\Omega u_\lambda^2\\
     &\quad-\frac{1}{2q}\int_\Omega\sum_{i,j=1}^N\left( D_sa_{ij}(x,u_\lambda)D_{x_i}u_\lambda D_{x_j}u_\lambda\right)u_\lambda.
 \end{align}
 From \eqref{a.2}-\eqref{a.3}, we obtain
 \begin{equation}\label{stima L_lambda inferiore}
 \begin{aligned}
     L_\lambda&\ge\frac{q-2-\gamma}{2q}\int_\Omega\sum_{i,j=1}^Na_{ij}(x,u_\lambda)D_{x_i}u_\lambda D_{x_j}u_\lambda+\frac{(q-2)\lambda}{2q}\int_\Omega u_\lambda^2\\
     &\ge\frac{(q-2-\gamma)\nu_0}{2q}\int_\Omega|\nabla u_\lambda|^2+\frac{(q-2)\lambda}{2q}\int_\Omega u_\lambda^2.
 \end{aligned}
 \end{equation}
 Hence, $\{u_\lambda\}_{\lambda}$ is uniformly bounded in $H^1(\Omega)$ by \eqref{ipotesi contraddizione Theorem 1.1}. Up to subsequence, $u_\lambda\wto u_\infty$ in $H^1(\Omega)$ as $\lambda\to+\infty$. Furthermore,
 \[
 \frac{q-2}{2q}\int_\Omega u_\lambda^2\le \frac{L_\lambda}{\lambda}\le\frac{M}{\lambda}\to0\quad \text{as $\lambda\to+\infty$}.
 \]
 Thus, $u_\lambda\to0$ in $L^2(\Omega)$ and $u_\infty\equiv0$. Up to  a further subsequence, we can assume that $u_\lambda\to 0$ in $L^r(\Omega)$ for every $r\in[1,2^*)$ and that $g(u_\lambda)\to 0$ in $\left(H^1(\Omega)\right)'$. Therefore, from \eqref{a.2}-\eqref{a.3} and Theorem \ref{brezis browder},
\begin{equation}\label{limite norma lambda}
 \begin{aligned}
     \nu_0\int_\Omega|\nabla u_\lambda|^2+\lambda\int_\Omega u_\lambda^2&\le\int_\Omega \sum_{i,j=1}^Na_{ij}(x,u_\lambda)D_{x_i}u_\lambda D_{x_j}u_\lambda+\frac12\int_\Omega\sum_{i,j=1}^N\left(D_sa_{ij}(x,u_\lambda)D_{x_i}u_\lambda D_{x_j}u_\lambda\right)u_\lambda \\
     &\quad+\lambda\int_\Omega u_\lambda^2=\int_\Omega g(u_\lambda)u_\lambda\to0\quad\text{as $\lambda\to+\infty$},
 \end{aligned}
 \end{equation}
 and $u_\lambda\to0$ strongly in $H^1(\Omega)$.

By Theorem \ref{brezis browder}, from \eqref{g.1} we get
 \begin{equation}\label{Stima epsilon superiore}
 \nu_0\int_\Omega|\nabla u_\lambda|^2+(\lambda-\varepsilon)\int_\Omega u_\lambda^2\le C_\varepsilon\int_\Omega |u_\lambda|^p,\qquad\text{for every $\varepsilon>0$.}
 \end{equation}
Moreover, by Sobolev's embedding, there exists a constant $d_0>0$ such that
\begin{equation}\label{sobolev d_0}
\int_\Omega|u_\lambda|^p\le d_0\left(\nu_0\int_\Omega|\nabla u_\lambda|^2+\int_\Omega u_\lambda^2\right)^{\frac p2}.
\end{equation}
Thus, by \eqref{Stima epsilon superiore} and \eqref{sobolev d_0} there exists $\widetilde d_0>0$ such that
\begin{align}
    \nu_0\int_\Omega|\nabla u_\lambda|^2+(\lambda-\varepsilon)\int_\Omega u_\lambda^2&\le \widetilde d_0\left(\nu_0\int_\Omega|\nabla u_\lambda|^2+\int_\Omega u_\lambda^2\right)^{\frac p2}\\
    &\le \widetilde d_0\left(\nu_0\int_\Omega|\nabla u_\lambda|^2+(\lambda-\varepsilon)\int_\Omega u_\lambda^2\right)^{\frac p2},
\end{align}
which implies
\[
\left(\nu_0\int_\Omega|\nabla u_\lambda|^2+(\lambda-\varepsilon)\int_\Omega u_\lambda^2\right)\ge d_1,
\]
for some $d_1>0$ independent of $\lambda$. 

{
On the other hand,
by \eqref{limite norma lambda} and for every fixed $\varepsilon>0$, we have
\[
\nu_0\int_\Omega |\nabla u_\lambda|^2
+(\lambda-\varepsilon)\int_\Omega u_\lambda^2
\longrightarrow 0,
\]
which contradicts the previous lower bound
\[
\nu_0\int_\Omega |\nabla u_\lambda|^2
+(\lambda-\varepsilon)\int_\Omega u_\lambda^2
\ge d_1>0.
\]
{Thus, $L_\lambda\to+\infty$ as $\lambda\to+\infty$.}

}

 }

\end{proof}

\begin{proof}[Proof of Corollary \ref{corollario lambda positivo}]
The conclusion follows directly from Theorem \ref{theorem lambda positive} by taking 
\[
g(s)=|s|^{p-2}s.
\]
It is straightforward to verify that assumptions \eqref{g.1'}--\eqref{g.2} are satisfied in this case. Indeed, since $G(s)=\frac{|s|^p}{p}$, condition \eqref{g.2} holds with $q=p$. Therefore the additional condition in Theorem \ref{theorem lambda positive}, namely
\[
\frac{p}{p-2}>\beta_p={\frac{p}{p-2}-\frac{N}{2}},
\]
Hence, this additional assumption is automatically satisfied in the Lane--Emden case.

Moreover, arguing as in Step 2 of the proof of Theorem \ref{theorem lambda positive}, { we observe that the constant functions
\[
u \equiv \pm \lambda^{\frac{1}{p-2}}
\]
are solutions and satisfy
\[
J_\lambda\left(\pm \lambda^{\frac{1}{p-2}}\right)
= \frac{p-2}{2p}\lambda^{\frac{p}{p-2}}|\Omega|.
\]
Furthermore,
\[
L_\lambda < J_\lambda\left(\pm\lambda^{1/(p-2)}\right)
\]
for every $\lambda>\lambda_0$, where
}
{
\begin{equation}\label{lambda_0}
\lambda_0:=
\max\left\{
\Lambda,\,
\left(\frac{2pC_1}{(p-2)|\Omega|}\right)^{2/N}
\right\}.
\end{equation}
{Here $\Lambda$ is the threshold in
Proposition~\ref{stima energia}.}
}
\end{proof}

{

\subsection{Uniform boundedness and nonexistence of nonconstant least energy solutions}

\begin{lemma}\label{lem:small-ground-state}
Assume that \eqref{a.1}-\eqref{a.3}, \eqref{g.1'}-\eqref{g.2} hold.
Let $\lambda_n\to0^+$ and, for every $n$, let $u_n$ be a least energy solution of
\[
\tag{$\mathcal P_{\lambda_n}$}\label{P_lambdan}
\begin{cases}
-\displaystyle\sum_{i,j=1}^N D_{x_j}
\left(a_{ij}(x,u)D_{x_i}u\right)
+
\displaystyle\frac12\sum_{i,j=1}^N
D_sa_{ij}(x,u)D_{x_i}uD_{x_j}u
=
g(u)-\lambda_n u
&\text{in }\Omega,
\\[2mm]
\displaystyle\sum_{i,j=1}^N
a_{ij}(x,u)D_{x_i}u\,\nu_j=0
&\text{on }\partial\Omega.
\end{cases}
\]
Then
\(
u_n\to0\ \text{in }H^1(\Omega).
\)
\end{lemma}

\begin{proof}
Set
\[
h(s):=\frac{g(s)}s,\qquad s>0.
\]
By \eqref{g.1'}, $h(s)\to0$ as $s\to0^+$, while by \eqref{g.2}, $h(s)>0$ for every $s>0$. Fix $s_0>0$. For $n$ sufficiently large, $0<\lambda_n<h(s_0)$, and hence there exists $c_n\in(0,s_0)$ such that
\[
h(c_n)=\lambda_n.
\]
Moreover, $c_n\to0$. Otherwise, up to a subsequence, $c_n\ge\eta>0$, and the continuity and positivity of $h$ would give
\[
\lambda_n=h(c_n)\ge\min_{[\eta,s_0]}h>0,
\]
which is impossible.

The constant functions $c_n$ are solutions of \eqref{P_lambdan}. Hence, by the definition of the least energy level,
\begin{equation}\label{stima livello 1}
J_{\lambda_n}(u_n)=L_{\lambda_n}\le J_{\lambda_n}(c_n)
=
|\Omega|\left(\frac{\lambda_n}{2}c_n^2-G(c_n)\right)
\le
\frac{|\Omega|}{2}\lambda_n c_n^2=o(\lambda_n).
\end{equation}

Since $u_n$ is a weak solution, by Theorem \ref{brezis browder} we may test the equation with $u_n$ and obtain
\begin{align*}
&\int_\Omega\sum_{i,j=1}^N a_{ij}(x,u_n)D_{x_i}u_nD_{x_j}u_n
+\frac12\int_\Omega\sum_{i,j=1}^N D_sa_{ij}(x,u_n)D_{x_i}u_nD_{x_j}u_n\,u_n
+\lambda_n\int_\Omega u_n^2
=
\int_\Omega g(u_n)u_n.
\end{align*}
By \eqref{g.2},
\[
\int_\Omega G(u_n)
\le
\frac1q\int_\Omega g(u_n)u_n.
\]
{
Arguing as in Step 3 of the proof of
Theorem~\ref{theorem lambda positive}, we obtain
\[
J_{\lambda_n}(u_n)
\ge
\frac{(q-2-\gamma)\nu_0}{2q}
\int_\Omega|\nabla u_n|^2
+
\frac{q-2}{2q}\lambda_n
\int_\Omega u_n^2.
\]
}
Combining this estimate with \eqref{stima livello 1}, we obtain
\[
\int_\Omega|\nabla u_n|^2=o(\lambda_n),\qquad
\int_\Omega u_n^2=o(1).
\]
In particular, $u_n\to0$ in $H^1(\Omega)$.
\end{proof}

\begin{lemma}\label{lem:uniform-Moser}
Assume that \eqref{a.1}--\eqref{a.3}, \eqref{g.1'} hold and let $\lambda>0$.
Let $u\in H^1(\Omega)\cap L^\infty(\Omega)$ be a weak solution of \eqref{P}. Assume moreover that
\[
\|u\|_{H^1(\Omega)}\le M
\]
for some $M>0$. Then there exists a constant $C=C(M)>0$, independent of $u$ and $\lambda$, such that
\[
\|u\|_{L^\infty(\Omega)}
\le C\|u\|_{L^2(\Omega)}.
\]
\end{lemma}

\begin{proof}
By Theorem~\ref{regolarità}, $u\in L^\infty(\Omega)$. Hence, for every $s\ge0$, the function
\(
u|u|^{2s}
\)
is an admissible test function in \eqref{P}. We obtain
\begin{align*}
&(2s+1)\int_\Omega\sum_{i,j=1}^N a_{ij}(x,u)|u|^{2s}D_{x_i}uD_{x_j}u
+\frac12\int_\Omega\sum_{i,j=1}^N D_sa_{ij}(x,u)D_{x_i}uD_{x_j}u\,u|u|^{2s}\\
&\qquad
+\lambda\int_\Omega |u|^{2s+2}
=
\int_\Omega g(u)u|u|^{2s}.
\end{align*}
By \eqref{a.3} and since $\lambda>0$, the second and third terms on the left-hand side are nonnegative. Moreover, by \eqref{g.1}, with $\varepsilon=1$,
\[
|g(u)u||u|^{2s}
\le
\left(1+C|u|^{p-2}\right)|u|^{2s+2}.
\]
Therefore, by \eqref{a.2},
\[
\nu_0(2s+1)\int_\Omega |u|^{2s}|\nabla u|^2
\le
C\int_\Omega\left(1+|u|^{p-2}\right)|u|^{2s+2}.
\]
Set
\(
\psi:=u|u|^s.
\)
Then
\[
|\nabla\psi|^2=(s+1)^2|u|^{2s}|\nabla u|^2,
\]
and hence
\begin{equation}\label{Linfty iteration 1}
\nu_0\frac{2s+1}{(s+1)^2}\int_\Omega|\nabla\psi|^2
\le
C\int_\Omega\left(1+|u|^{p-2}\right)\psi^2.
\end{equation}

By the Sobolev inequality and \eqref{Linfty iteration 1},
\begin{align*}
\|\psi\|_{L^{2^*}(\Omega)}^2
&\le
C\left(\int_\Omega|\nabla\psi|^2+\int_\Omega\psi^2\right)\le
C(s+1)\int_\Omega\left(1+|u|^{p-2}\right)\psi^2.
\end{align*}
Here and in what follows, $C>0$ denotes a constant possibly depending on $M$, but independent of $u$, $\lambda$, and $s$.

Let
\(
r:=\frac{2^*}{p-2}.
\)
Since $p<2^*$, one has
\(
r>\frac N2.
\)
Moreover, since
\(
\|u\|_{H^1(\Omega)}\le M,
\)
the Sobolev embedding gives
\(
\|u\|_{L^{2^*}(\Omega)}\le CM.
\)
Thus
\[
\bigl\|1+|u|^{p-2}\bigr\|_{L^r(\Omega)}
\le C
\]
uniformly with respect to $u$ and $\lambda$. By H\"older's inequality,
\begin{equation}\label{Linfty iteration 2}
\|\psi\|_{L^{2^*}(\Omega)}^2
\le
C(s+1)\|\psi\|_{L^{2r'}(\Omega)}^2.
\end{equation}

Since $r>N/2$, we have
\(
2<2r'<2^*.
\)
Hence, setting
\(
\theta:=\frac{N}{2r}\in(0,1),
\)
the interpolation inequality gives
\[
\|\psi\|_{L^{2r'}(\Omega)}
\le
\|\psi\|_{L^{2^*}(\Omega)}^\theta
\|\psi\|_{L^2(\Omega)}^{1-\theta}.
\]
Combining this estimate with \eqref{Linfty iteration 2}, we obtain
\[
\|\psi\|_{L^{2^*}(\Omega)}^2
\le
C(s+1)
\|\psi\|_{L^{2^*}(\Omega)}^{2\theta}
\|\psi\|_{L^2(\Omega)}^{2(1-\theta)}.
\]
Therefore,
\begin{equation}\label{Linfty iteration 3}
\|\psi\|_{L^{2^*}(\Omega)}
\le
\bigl(C(s+1)\bigr)^{\frac1{2(1-\theta)}}
\|\psi\|_{L^2(\Omega)}.
\end{equation}

We now perform a standard Moser iteration. Set
\[
q_0:=2(s+1),\qquad \tau:=\frac{2^*}{2}=\frac{N}{N-2}>1.
\]
Since
\(
\|\psi\|_{L^2(\Omega)}
=
\|u\|_{L^{q_0}(\Omega)}^{q_0/2}
\)
and
\(
\|\psi\|_{L^{2^*}(\Omega)}
=
\|u\|_{L^{\tau q_0}(\Omega)}^{q_0/2},
\)
inequality \eqref{Linfty iteration 3} yields
\[
\|u\|_{L^{\tau q_0}(\Omega)}
\le
(Cq_0)^{\frac1{(1-\theta)q_0}}
\|u\|_{L^{q_0}(\Omega)}
\qquad\text{for every }q_0\ge2.
\]
Taking
\(
q_k:=2\tau^k,
\)
we obtain
\[
\|u\|_{L^{q_{k+1}}(\Omega)}
\le
(Cq_k)^{\frac1{(1-\theta)q_k}}
\|u\|_{L^{q_k}(\Omega)}.
\]
Iterating,
\[
\|u\|_{L^{q_{k+1}}(\Omega)}
\le
\left[
\prod_{j=0}^k
(Cq_j)^{\frac1{(1-\theta)q_j}}
\right]
\|u\|_{L^2(\Omega)}.
\]
Since
\[
\sum_{j=0}^\infty\frac1{q_j}<+\infty
\qquad\text{and}\qquad
\sum_{j=0}^\infty\frac{\log q_j}{q_j}<+\infty,
\]
the product on the right-hand side is bounded independently of $k$. Passing to the limit as $k\to+\infty$, we conclude that
\[
\|u\|_{L^\infty(\Omega)}
\le
C\|u\|_{L^2(\Omega)}.
\]
\end{proof}

\begin{corollary}\label{lem:H1-Linfty-smallness}
Assume that \eqref{a.1}--\eqref{a.3}, \eqref{g.1'} hold.
Let $\lambda_n\to0^+$ and let $u_n$ be weak solutions of \eqref{P_lambdan} such that
\(
u_n\to0
\ \text{in }H^1(\Omega).
\)
Then
\(
u_n\to0
\ \text{in }L^\infty(\Omega).
\)
\end{corollary}

\begin{proof}
Since $u_n\to0$ in $H^1(\Omega)$, the sequence $\{u_n\}$ is bounded in $H^1(\Omega)$. Hence Lemma~\ref{lem:uniform-Moser} gives
\[
\|u_n\|_{L^\infty(\Omega)}
\le C\|u_n\|_{L^2(\Omega)}.
\]
Since $u_n\to0$ in $L^2(\Omega)$, the conclusion follows.
\end{proof}

\begin{lemma}\label{lem:small-solution-rigidity}
Assume that \eqref{a.1}-\eqref{a.2}, \eqref{g.1'} hold and let $g\in C^1(\mathbb R)$. There exists $\delta>0$ such that, for every $\lambda>0$, every weak solution
\(
u\in H^1(\Omega)\cap L^\infty(\Omega)
\)
of \eqref{P} satisfying
\(
\|u\|_{L^\infty(\Omega)}\le\delta
\)
is constant.
\end{lemma}

\begin{proof}
Set
\[
c:=\frac1{|\Omega|}\int_\Omega u,\qquad w:=u-c.
\]
Since $u\in L^\infty(\Omega)$, we have $w\in H^1(\Omega)\cap L^\infty(\Omega)$. Hence $w$ is an admissible test function. Testing \eqref{P} with $w$, we obtain
\begin{align*}
&\int_\Omega\sum_{i,j=1}^N a_{ij}(x,u)D_{x_i}wD_{x_j}w
+\frac12\int_\Omega\sum_{i,j=1}^N D_sa_{ij}(x,u)D_{x_i}wD_{x_j}w\,w
+\lambda\int_\Omega w^2
=
\int_\Omega g(u)w.
\end{align*}
Since $\int_\Omega w=0$,
\[
\int_\Omega g(u)w
=
\int_\Omega\left(g(u)-g(c)\right)w.
\]
By \eqref{g.1'}, $g(0)=g'(0)=0$. Set
\(
\omega(\rho):=\max_{|s|\le\rho}|g'(s)|.
\)
Since $g\in C^1(\mathbb R)$, $\omega(\rho)\to0$ as $\rho\to0^+$. Moreover, observing that
\(
|c|\le\|u\|_{L^\infty(\Omega)},
\)
the mean value theorem gives
\[
\left|\int_\Omega g(u)w\right|
\le
\omega\left(\|u\|_{L^\infty(\Omega)}\right)\int_\Omega w^2.
\]
On the other hand, by \eqref{a.1},
\begin{align*}
\left|
\frac12\int_\Omega\sum_{i,j=1}^N D_sa_{ij}(x,u)D_{x_i}wD_{x_j}w\,w
\right|
&\le
\frac{C_0}{2}\|w\|_{L^\infty(\Omega)}
\int_\Omega\sum_{i,j=1}^N|D_{x_i}w||D_{x_j}w|\\
&\le
\frac{NC_0}{2}\|w\|_{L^\infty(\Omega)}
\int_\Omega|\nabla w|^2.
\end{align*}
Since
\(
\|w\|_{L^\infty(\Omega)}
\le2\|u\|_{L^\infty(\Omega)}
\)
, by \eqref{a.2}, we have
\[
\left(
\nu_0-NC_0\|u\|_{L^\infty(\Omega)}
\right)
\int_\Omega|\nabla w|^2
+
\lambda\int_\Omega w^2
\le
\omega\left(\|u\|_{L^\infty(\Omega)}\right)\int_\Omega w^2.
\]
{As} $w$ has zero mean, the Poincar\'e inequality gives
\[
\int_\Omega w^2
\le
C_P\int_\Omega|\nabla w|^2.
\]
Therefore,
\[
\left[
\nu_0
-
NC_0\|u\|_{L^\infty(\Omega)}
-
C_P\omega\left(\|u\|_{L^\infty(\Omega)}\right)
\right]
\int_\Omega|\nabla w|^2
\le0.
\]
Choose $\delta>0$ so small that
\(
\nu_0-NC_0\delta-C_P\omega(\delta)>0.
\)
If $\|u\|_{L^\infty(\Omega)}\le\delta$, it follows that
\[
\int_\Omega|\nabla w|^2=0.
\]
Since $w$ has zero mean, $w\equiv0$. Hence $u\equiv c$.
\end{proof}

\begin{proof}[Proof of Proposition~\ref{small lambda general}]
Assume by contradiction that there exist $\lambda_n\to0^+$ and nonconstant least energy solutions $u_n$ of \eqref{P_lambdan}. By Lemma~\ref{lem:small-ground-state},
\[
u_n\to0\qquad\text{in }H^1(\Omega).
\]
Hence, by Corollary~\ref{lem:H1-Linfty-smallness},
\[
u_n\to0\qquad\text{in }L^\infty(\Omega).
\]
Therefore, for $n$ sufficiently large,
\[
\|u_n\|_{L^\infty(\Omega)}\le\delta,
\]
where $\delta$ is given by Lemma~\ref{lem:small-solution-rigidity}. It follows that $u_n$ is constant, a contradiction. Thus every least energy solution is constant for all sufficiently small $\lambda>0$. If $u_\lambda\equiv c_\lambda\neq0$, then
\(
g(c_\lambda)=\lambda c_\lambda,
\)
and consequently
\[
L_\lambda=J_\lambda(c_\lambda)=|\Omega|\left(\frac{\lambda}{2}c_\lambda^2-G(c_\lambda)\right).
\]
\end{proof}

\begin{proof}[Proof of Corollary~\ref{small lambda power}]
It is enough to apply Proposition~\ref{small lambda general} with
\(
g(s)=|s|^{p-2}s.
\)
Indeed, $g\in C^1(\mathbb R)$ and \eqref{g.1'}-\eqref{g.2} hold with $q=p$. If $u_\lambda\equiv c\neq0$ is a constant solution of \eqref{lane emden lambda problem}, then
\(
|c|^{p-2}c=\lambda c,
\)
and hence
\(
c=\pm\lambda^{\frac1{p-2}}.
\)
Consequently,
\[
L_\lambda=J_\lambda\left(\pm\lambda^{\frac1{p-2}}\right)=\frac{p-2}{2p}|\Omega|\lambda^{\frac{p}{p-2}}.
\]
\end{proof}

}

 \section{Indefinite and non-resonant case}

{
Consider the following elliptic eigenvalue problem with Neumann boundary conditions:
\begin{equation}\label{eigenvalue problem}
\left\{
\begin{aligned}
    -\sum_{i,j=1}^N D_{x_j}(a_{ij}(x,0)D_{x_i}\varphi_k) &= \lambda_k \varphi_k \quad &\text{in } \Omega,\\
    \sum_{i,j=1}^N a_{ij}(x,0)D_{x_i}\varphi_k\, \nu_j &= 0 \quad &\text{on } \partial\Omega.
\end{aligned}
\right.
\end{equation}

Assume that the coefficients $a_{ij}(x,0)$ satisfy \eqref{a.1} and \eqref{a.2}, namely boundedness and uniform ellipticity. 

It is well known 
that there exists a sequence  $\{(\lambda_k,\varphi_k)\}_{k\ge1}$ such that $\lambda_k \ge 0$, $\varphi_k \in H^1(\Omega)$, and each pair $(\lambda_k,\varphi_k)$ solves \eqref{eigenvalue problem}. Moreover, the eigenfunctions can be chosen to be orthonormal in $L^2(\Omega)$ and orthogonal with respect to the energy bilinear form, namely
\[
\int_\Omega \sum_{i,j=1}^N a_{ij}(x,0) D_{x_i}\varphi_\ell D_{x_j}\varphi_k 
= \lambda_\ell \int_\Omega \varphi_\ell \varphi_k 
= \lambda_\ell \delta_{\ell k}=
\begin{cases}
    \lambda_\ell&\text{if $\ell=k$}\\
    0&\text{if $\ell\ne k$.}
\end{cases}
\]

Furthermore, $\lambda_k \to +\infty$ as $k \to +\infty$, and the family $\{\varphi_k\}$ is complete (in particular, it forms an orthonormal basis of $L^2(\Omega)$ and generates $H^1(\Omega)$).

Due to the Neumann boundary conditions, the first eigenvalue is $\lambda_1 = 0$, and the corresponding eigenfunction is constant:
\[
\varphi_1 = |\Omega|^{-1/2}.
\]

}

\begin{proof}[Proof of Theorem \ref{theorem lambda<0 non resonant case}]
{
The multiplicity and nonconstancy assertions follow from Theorem~\ref{known existence}. Assume now \eqref{non-resonant hypothesis}. Choose nontrivial solutions $u_n$ with $J_\lambda(u_n)\to L_\lambda$. Theorem~\ref{PS condition} gives, along a subsequence, $u_n\to u$ strongly in $H^1(\Omega)$, where $u$ solves \eqref{P} and $J_\lambda(u)=L_\lambda$. We will prove that $u\ne0$.
}

Assume by contradiction that $u_n \to 0$ in $H^1(\Omega)$ and set
\[
z_n := \frac{u_n}{\|u_n\|}.
\]
Up to a subsequence, we may assume that $z_n \rightharpoonup z$ in $H^1(\Omega)$. 
Testing \eqref{P} with $\varphi \in H^1(\Omega)\cap L^\infty(\Omega)$ and dividing by $\|u_n\|$, we get
\begin{equation}\label{equazione normalizzata z_n}
\int_\Omega\sum_{i,j=1}^N a_{ij}(x,u_n)D_{x_i} z_n D_{x_j} \varphi
+\lambda\int_\Omega z_n \varphi
+\frac12\int_\Omega \sum_{i,j=1}^N\left( D_s a_{ij}(x,u_n)D_{x_i} u_n D_{x_j} z_n\right) \varphi
= \int_\Omega \frac{g(\|u_n\|z_n)}{\|u_n\|}\varphi.
\end{equation}

Since $g(s)=o(s)$ as $s\to0$ by \eqref{g.1'}, it follows that
\begin{equation}\label{limite z_n}
\lim_{n\to+\infty}\left\{\int_\Omega\sum_{i,j=1}^N a_{ij}(x,u_n)D_{x_i} z_n D_{x_j} \varphi
+\lambda\int_\Omega z_n \varphi\right\}=0.
\end{equation}

Passing to the limit, using \eqref{a.1}--\eqref{a.2} and the weak convergence of $z_n$, we obtain
\[
\int_\Omega\sum_{i,j=1}^N a_{ij}(x,0)D_{x_i} z D_{x_j} \varphi
+\lambda\int_\Omega z \varphi = 0.
\]
By the non-resonance assumption \eqref{non-resonant hypothesis}, this implies $z \equiv 0$.

{
Taking $\varphi=z_n$ in \eqref{equazione normalizzata z_n}, we obtain
\[
\begin{aligned}
&\int_\Omega \sum_{i,j=1}^N
a_{ij}(x,u_n)D_{x_i}z_nD_{x_j}z_n
+\lambda\int_\Omega z_n^2 
+\frac12\int_\Omega\sum_{i,j=1}^N
\bigl(D_sa_{ij}(x,u_n)D_{x_i}z_nD_{x_j}z_n\bigr)u_n
=
\int_\Omega \frac{g(u_n)}{\|u_n\|}z_n .
\end{aligned}
\]
By \eqref{a.3}, the third term on the left-hand side is nonnegative.
Moreover, from \eqref{g.1},
\[
\begin{aligned}
\left|
\int_\Omega \frac{g(u_n)}{\|u_n\|}z_n
\right|
&\le
\varepsilon\int_\Omega z_n^2
+
C_\varepsilon \|u_n\|^{p-2}
\int_\Omega |z_n|^p =o(1),
\end{aligned}
\]
since $\{z_n\}$ is bounded in $H^1(\Omega)$, $p>2$, and
$\|u_n\|\to0$. Therefore, recalling that $\lambda\le0$ and
$z_n\to0$ in $L^2(\Omega)$, by $(a.2)$ we get
\[
\nu_0\int_\Omega |\nabla z_n|^2
\le
-\lambda\int_\Omega z_n^2+o(1)
=o(1).
\]
Hence $z_n\to0$ strongly in $H^1(\Omega)$, which contradicts
$\|z_n\|=1$. Thus $u\not\equiv0$, and the least energy level
$L_\lambda$ is achieved.
}




\end{proof}

\section{Indefinite and resonant case}
In this section, we consider the problem \eqref{lane emden problem}:
    \begin{equation}
    \left\{
    \begin{aligned}
-&\sum_{i,j=1}^ND_{x_j}(a_{ij}(x,u)D_{x_i}u)+\frac12\sum_{i,j=1}^ND_sa_{ij}(x,u)D_{x_i}uD_{x_j}u=|u|^{p-2}u\quad \text{in $\Omega$},\\
&\sum_{i,j=1}^Na_{ij}(x,u)D_{x_i}u\ \nu_j=0\quad\text{on $\partial\Omega$},
\end{aligned}
\right.
\end{equation}

and the respective energy functional $J_0:H^1(\Omega)\to\R$, given by
\[
J_0(u)=\frac12\int_\Omega\sum_{i,j=1}^Na_{ij}(x,u)D_{x_i}u D_{x_j}u-\frac{1}p\int_\Omega |u|^p,\qquad p\in(2,2^*).
\]
We denote by $H^0\subset H^1(\Omega)$ the subspace of constant functions.

Define the least energy level
\[
L_0=\inf\left\{J_0(u)\ :\ \text{$u\not\equiv0$ is a weak solution of \eqref{lane emden problem}}\right \}.
\]

{
\begin{proof}[Proof of Theorem \ref{lane emden theorem}]
\textbf{Step 1}: existence of a nonconstant least energy solution.

{
Theorem~\ref{known existence}, applied with $\lambda=0$ and $g(s)=|s|^{p-2}s$, give a nonempty solution set.
}

Let $\{u_n\}$ be a minimizing sequence of nontrivial solutions, namely $u_n\not\equiv0$, $u_n$ solves \eqref{lane emden problem}, and $J_0(u_n)\to L_0$.
{By Theorem~\ref{PS condition}, after passing to a subsequence, we have $u_n\to u$ in $H^1(\Omega)$, where $u$ is a weak solution and $J_0(u)=L_0$.
}

From \eqref{a.2} and Theorem \ref{brezis browder}, we have
\begin{equation}\label{stima gradiente}
\nu_0 \int_\Omega |\nabla u_n|^2 \le \int_\Omega |u_n|^p \le C_S \|u_n\|^p,
\end{equation}
for some constant $C_S>0$.

Assume by contradiction that $\|u_n\|\to 0$, and define
\[
z_n := \frac{u_n}{\|u_n\|}.
\]
Dividing \eqref{stima gradiente} by $\|u_n\|^2$, we get $\|\nabla z_n\|_2\to0$. Since $\|z_n\|_{H^1}=1$, the $L^2$-norm of $z_n$ does not vanish. Hence, up to a subsequence, $z_n \rightharpoonup z_\infty$ in $H^1(\Omega)$ and $z_n \to z_\infty$ in $L^p(\Omega)$, with $z_\infty$ constant and $z_\infty\neq0$.

Testing \eqref{lane emden problem} with $z_\infty \in C^\infty(\overline{\Omega})$, we obtain
\[
\int_\Omega \sum_{i,j=1}^N \left(D_s a_{ij}(x,u_n) D_{x_i}u_n D_{x_j}u_n \right) z_\infty
= 2 \int_\Omega |u_n|^{p-2} u_n z_\infty.
\]
Dividing both sides by $\|u_n\|^{p-1}$ yields
\[
\|u_n\|^{1-p} \int_\Omega \sum_{i,j=1}^N\left( D_s a_{ij}(x,u_n) D_{x_i}u_n D_{x_j}u_n \right) z_\infty
= 2 \int_\Omega |z_n|^{p-2} z_n z_\infty.
\]

Passing to the limit as $n \to +\infty$, we obtain
\[
2 \int_\Omega |z_\infty|^p 
= \lim_{n \to +\infty} 
\|u_n\|^{1-p} \int_\Omega \sum_{i,j=1}^N\left( D_s a_{ij}(x,u_n) D_{x_i}u_n D_{x_j}u_n \right) z_\infty.
\]

On the other hand, using \eqref{a.1} and \eqref{a.2}, we estimate
\[
\left| \int_\Omega \sum_{i,j=1}^N \left(D_s a_{ij}(x,u_n) D_{x_i}u_n D_{x_j}u_n \right) z_\infty \right|
\le NC_0 \int_\Omega |\nabla u_n|^2 |z_\infty|
\le \frac{NC_0}{\nu_0} \int_\Omega |u_n|^p |z_\infty|.
\]
Therefore,
\[
2 \int_\Omega |z_\infty|^p 
\le \liminf_{n \to +\infty} 
\frac{NC_0}{\nu_0} \|u_n\|^{1-p} \int_\Omega |u_n|^p |z_\infty|
\le \liminf_{n \to +\infty} 
\frac{NC_0 C_S}{\nu_0} \|u_n\| \int_\Omega |z_\infty|.
\]

Since $\|u_n\|\to 0$, the right-hand side tends to $0$, yielding a contradiction with
\[
\int_\Omega |z_\infty|^p > 0.
\]

Hence, $\|u_n\| \not\to 0$, and therefore there exists a nontrivial least energy solution 
$u \in H^1(\Omega)\setminus\{0\}$.

Moreover, Theorem \ref{regolarità} implies that $u\in H^1(\Omega)\cap L^\infty(\Omega)$.

Since $u$ is a solution of \eqref{lane emden problem}, we have
\[
\int_\Omega|u|^{p-2}u=\frac12\int_\Omega\sum_{i,j=1}^ND_sa_{ij}(x,u)D_{x_i}u D_{x_j}u.
\]
Thus, $u$ is not a constant.

\textbf{Step 2}: $L_\lambda\le J_\lambda(\lambda^{\frac{1}{p-2}})\to 0$ as $\lambda\to0^+$. 

For every $\lambda>0$, the constant functions
\[
\kappa_\lambda\equiv \pm \lambda^{\frac{1}{p-2}}
\]
solve \eqref{lane emden lambda problem}. Therefore, by the definition of $L_\lambda$,
\[
L_\lambda\le J_\lambda(\kappa_\lambda).
\]
Since $\nabla \kappa_\lambda=0$, we have
\[
J_\lambda(\kappa_\lambda)
=\frac{\lambda}{2}\int_\Omega \kappa_\lambda^2-\frac1p\int_\Omega |\kappa_\lambda|^p
=\left(\frac12-\frac1p\right)|\Omega|\lambda^{\frac{p}{p-2}}
=\frac{p-2}{2p}|\Omega|\lambda^{\frac{p}{p-2}}.
\]
Hence $L_\lambda\to0$ as $\lambda\to0^+$.

By \eqref{a.3} we have
\begin{align}
L_0&=\frac{p-2}{2p}\int_\Omega \sum_{i,j=1}^Na_{ij}(x,u)D_{x_i}u D_{x_j}u-\frac{1}{2p}\int_\Omega\left( D_sa_{ij}(x,u) D_{x_i}uD_{x_j}u\right) u\\
&\ge\frac{p-2-\gamma}{2p}\int_\Omega\sum_{i,j=1}^Na_{ij}(x,u)D_{x_i}u D_{x_j}u>0.
\end{align}
Then $L_\lambda:[0,+\infty)\to(0,+\infty)$ is not continuous at $\lambda=0$.

\end{proof}
}

\begin{proof}[Proof of Corollary \ref{semilinear case}]
We divide the proof into three steps.

\medskip

{
\textbf{Step 1}: The case $\lambda\to0^+$ and $\lambda=0$.
Since the coefficients are independent of $s$, condition \eqref{a.3} is trivially satisfied. Hence Corollary~\ref{small lambda power} applies and yields the existence of $\widetilde\lambda_0>0$ such that, for every $0<\lambda<\widetilde\lambda_0$, every least energy solution is constant. Therefore,
\[
u_\lambda\equiv\pm\lambda^{\frac1{p-2}},
\qquad
L_\lambda=\frac{p-2}{2p}|\Omega|\lambda^{\frac{p}{p-2}}.
\]
In particular,
\[
L_\lambda\to0
\qquad\text{as }\lambda\to0^+.
\]
On the other hand, Theorem~\ref{lane emden theorem} gives the existence of a nonconstant least energy solution for $\lambda=0$, with $L_0>0$.
}
\medskip

{\noindent\textbf{Step 2}:} The case $\lambda>0$ and the limit as $\lambda\to+\infty$.
The existence of a least energy solution $u_\lambda \in H^1(\Omega)\cap L^\infty(\Omega)$ for every $\lambda>0$ follows from Theorem \ref{theorem lambda positive}. Moreover, the same result ensures that
\[
L_\lambda \to +\infty \quad \text{as } \lambda \to +\infty,
\]
and that $L_\lambda$ is achieved by a nonconstant solution for $\lambda$ sufficiently large. {By the threshold $\lambda_0$ defined in
\eqref{lambda_0}, every least energy solution is
nonconstant for $\lambda>\lambda_0$.}

\medskip

\noindent\textbf{Step 3}: The case $\lambda<0$.
We now prove the existence of a least energy solution for every $\lambda<0$.

{
Theorem~\ref{known existence} supplies nontrivial solutions. Since the coefficients are independent of $s$, their evenness is automatic, and the power nonlinearity is odd. 
} 

{
Let $u_n$ be nontrivial solutions with $J_\lambda(u_n)\to L_\lambda$. By Theorem~\ref{PS condition}, a subsequence converges strongly in $H^1(\Omega)$ to a weak solution $u$ with $J_\lambda(u)=L_\lambda$. We claim that $u\ne0$.
}

Assume by contradiction that $u_n \to 0$ in $H^1(\Omega)$ and set
\[
z_n := \frac{u_n}{\|u_n\|}.
\]
Up to a subsequence, we may assume that $z_n \rightharpoonup z$ in $H^1(\Omega)$.

Let us consider the quadratic form
\[
Q_\lambda(u) := \int_\Omega \sum_{i,j=1}^N a_{ij}(x) D_{x_i}u D_{x_j}u + \lambda \int_\Omega u^2.
\]

By the spectral decomposition associated with the eigenvalue problem \eqref{eigenvalue problem}, we define
\[
H^+ := \overline{\mathrm{span}\{\varphi_k : \lambda_k + \lambda > 0\}}, \quad
H^- := \mathrm{span}\{\varphi_k : \lambda_k + \lambda < 0\}, \quad
H^0 := \mathrm{span}\{\varphi_k : \lambda_k + \lambda = 0\}.
\]

Then we have the orthogonal decomposition
\[
H^1(\Omega) = H^+ \oplus H^- \oplus H^0,
\]
where the orthogonality is understood in $L^2(\Omega)$.

Moreover, the quadratic form $Q_\lambda$ is positive definite on $H^+$, negative definite on $H^-$, and vanishes on $H^0$.

Accordingly, for every $u \in H^1(\Omega)$, we can write
\[
u = u^+ + u^- + u^0,
\quad \text{with } u^\pm \in H^\pm, \; u^0 \in H^0.
\]

In particular, for the sequence $\{u_n\}$ and the normalized sequence $\{z_n\}$, we have
\[
u_n = u_n^+ + u_n^- + u_n^0, \qquad
z_n = z_n^+ + z_n^- + z_n^0.
\]

On the other hand, $u_n$ satisfies
\[
\int_\Omega\sum_{i,j=1}^N a_{ij}(x)D_{x_i}u_n D_{x_j}\varphi
+\lambda\int_\Omega u_n \varphi
=\int_\Omega |u_n|^{p-2}u_n \varphi,
\qquad \text{for every } \varphi \in H^1(\Omega).
\]
Dividing by $\|u_n\|$, we obtain
\begin{equation}\label{equazione z_n semilineare}
\int_\Omega\sum_{i,j=1}^N a_{ij}(x)D_{x_i}z_n D_{x_j}\varphi
+\lambda\int_\Omega z_n \varphi
=\|u_n\|^{p-2}\int_\Omega |z_n|^{p-2}z_n \varphi.
\end{equation}

Since $z_n \rightharpoonup z$ in $H^1(\Omega)$, $\|u_n\|\to0$, and $p>2$, passing to the limit in \eqref{equazione z_n semilineare} yields
\[
\int_\Omega\sum_{i,j=1}^N a_{ij}(x)D_{x_i}z D_{x_j}\varphi
+\lambda\int_\Omega z \varphi = 0,
\]
for every $\varphi \in H^1(\Omega)$. Hence, $z \in H^0$.

\medskip

Taking $\varphi = z_n^+$ in \eqref{equazione z_n semilineare} and using the orthogonality of the decomposition, we obtain
\[
\int_\Omega\sum_{i,j=1}^N a_{ij}(x)D_{x_i}z_n^+ D_{x_j}z_n^+
+\lambda\int_\Omega |z_n^+|^2
=\|u_n\|^{p-2}\int_\Omega |z_n|^{p-2}z_n z_n^+.
\]
Since the right-hand side tends to $0$, and the quadratic form is positive definite on $H^+$, we deduce that
\[
z_n^+ \to 0 \quad \text{in } H^1(\Omega).
\]

In particular, recalling that $z_n \rightharpoonup z \in H^0$ and that $H^-$ is finite-dimensional, we conclude that
\[
z_n \to z \quad \text{strongly in } H^1(\Omega).
\]

\medskip

Next, taking $\varphi = z_n^0$ in \eqref{equazione z_n semilineare}, we obtain
\[
\|u_n\|^{p-2}\int_\Omega |z_n|^{p-2}z_n z_n^0
=\int_\Omega\sum_{i,j=1}^N a_{ij}(x)D_{x_i}z_n D_{x_j}z_n^0
+\lambda\int_\Omega z_n z_n^0.
\]
By orthogonality of the decomposition, we have
\[
\int_\Omega\sum_{i,j=1}^N a_{ij}(x)D_{x_i}z_n D_{x_j}z_n^0
+\lambda\int_\Omega z_n z_n^0
=\int_\Omega\sum_{i,j=1}^N a_{ij}(x)D_{x_i}z_n^0 D_{x_j}z_n^0
+\lambda\int_\Omega |z_n^0|^2 = 0.
\]
Therefore,
\[
\int_\Omega |z_n|^{p-2}z_n z_n^0 = 0.
\]

Passing to the limit and using the strong convergence of $z_n$, we obtain
\[
\int_\Omega |z|^p = 0,
\]
which implies $z \equiv 0$.

This contradicts $\|z_n\|=1$, and the proof is complete.

\end{proof}


\begin{thebibliography}{10}

\bibitem{arcoyaboccardo1}
David Arcoya and Lucio Boccardo.
\newblock Critical points for multiple integrals of the calculus of variations.
\newblock {\em Arch. Rational Mech. Anal.}, 134(3):249--274, 1996.


\bibitem{candela2006multiple}
Anna Maria Candela and Giuliana Palmieri.
\newblock  Multiple solutions of some nonlinear variational problems.
\newblock {\em Advanced Nonlinear Studies}, 6(2):269--286, 2006.


\bibitem{candela2009some}
Anna~Maria Candela and Giuliana Palmieri.
\newblock Some abstract critical point theorems and applications.
\newblock  {\em Discrete Continuous Dynamical Systems}, Suppl 2009, 133--142, 2009
  
    \bibitem{candela2009infinitely}
Anna~Maria Candela and Giuliana Palmieri.
\newblock Infinitely many solutions of some nonlinear variational equations.
\newblock {\em Calculus of Variations and Partial Differential Equations},
  34(4):495--530, 2009.

  \bibitem{caninoquasilineare1}
Annamaria Canino.
\newblock Multiplicity of solutions for quasilinear elliptic equations.
\newblock {\em Topol. Methods Nonlinear Anal.}, 6(2):357--370, 1995.

\bibitem{caninoserdica}
Annamaria Canino.
\newblock On a variational approach to some quasilinear problems.
\newblock {\em Serdica Math. J.}, 22(3):297--324, 1996.


\bibitem{nonsmooththeory1}
Annamaria Canino and Marco Degiovanni.
\newblock Nonsmooth critical point theory and quasilinear elliptic equations.
\newblock In {\em Topological methods in differential equations and inclusions
  ({M}ontreal, {PQ}, 1994)}, volume 472 of {\em NATO Adv. Sci. Inst. Ser. C:
  Math. Phys. Sci.}, pages 1--50. Kluwer Acad. Publ., Dordrecht, 1995.

 \bibitem{CaninoMauroNeumann2026}
Annamaria Canino and Simone Mauro. Quasilinear Equations with Neumann Boundary Conditions. \emph{Mediterr. J. Math.} 23, 44 (2026). 

\bibitem{CaninoMauroCoopCompSystem2025}
Annamaria Canino, and Simone Mauro.
Quasilinear Elliptic Cooperative and Competitive Systems.
\emph{arXiv preprint} arXiv:2510.18758 (2025).

\bibitem{corvellec1993deformation}
Jean-No\"el Corvellec, Marco Degiovanni, and Marco Marzocchi.
\newblock Deformation properties for continuous functionals and critical point
  theory.
\newblock {\em Topol. Methods Nonlinear Anal.}, 1(1):151--171, 1993.

\bibitem{degiovanni1994critical}
Marco Degiovanni and Marco Marzocchi.
\newblock A critical point theory for nonsmooth functional.
\newblock {\em Annali di Matematica Pura ed Applicata}, 167:73--100, 1994.

\bibitem{gilbarg2001elliptic}
David Gilbarg and Neil~S. Trudinger.
\newblock Elliptic partial differential equations of second order.
\newblock {\em Springer}, Berlin, 2001.

\bibitem{gidasspruck1981global}
Basilis Gidas and Joel Spruck.
\newblock Global and local behavior of positive solutions of nonlinear elliptic equations.
\newblock {\em Communications on Pure and Applied Mathematics}, 34(4):525--598, 1981.

{
\bibitem{kivshar2003optical}
Yuri S. Kivshar and Govind P. Agrawal.
\newblock \emph{Optical Solitons: From Fibers to Photonic Crystals}.
\newblock Academic Press, 2003.
}

\bibitem{kwong1989uniqueness}
Man Kam Kwong.
\newblock Uniqueness of positive solutions of $\Delta u-u+u^p=0$ in $\mathbb R^N$.
\newblock {\em Archive for Rational Mechanics and Analysis}, 105(3):243--266, 1989.

\bibitem{ladyzhenskaya1968linear}
Olga~A. Ladyzhenskaya and Nina~N. Uraltseva.
\newblock Linear and quasilinear elliptic equations.
\newblock {\em Academic Press}, New York, 1968.

\bibitem{Lieberman1988}
Gary~M.~Lieberman,
\newblock Boundary regularity for solutions of degenerate elliptic equations,
\newblock \emph{Nonlinear Analysis: Theory, Methods \& Applications}, 
12 (1988), no.~11, 1203--1219.

\bibitem{lin1988large}
Chang-Shou Lin, Wei-Ming Ni, and Izumi Takagi.
\newblock Large amplitude stationary solutions to a chemotaxis system.
\newblock {\em Journal of Differential Equations}, 72(1):1--27, 1988.

\bibitem{lin2006diffusion}
Chang-Shou Lin and Wei-Ming Ni.
\newblock On the diffusion coefficient of a semilinear Neumann problem.
\newblock In {\em Calculus of Variations and Partial Differential Equations:
  Proceedings of a Conference held in Trento, Italy June 16--21, 1986}, pages
  160--174. Springer, 2006.

  \bibitem{nardi2015schauder}
Giacomo Nardi.
\newblock Schauder estimate for solutions of poisson's equation with Neumann
  boundary condition.
\newblock {\em L'enseignement Math{\'e}matique}, 60(3):421--435, 2015.



\bibitem{ni1991shape}
Wei-Ming Ni and I. Takagi.
\newblock On the shape of least-energy solutions to a semilinear Neumann problem.
\newblock {\em Communications on Pure and Applied Mathematics}, 44(7):819--851, 1991.

\bibitem{ni1993peaks}
Wei-Ming Ni and Izumi Takagi.
\newblock Locating the peaks of least-energy solutions to a semilinear Neumann problem.
\newblock {\em Duke Mathematical Journal}, 70(2):247--281, 1993.

{
\bibitem{pitaevskii2016bose}
Lev Pitaevskii and Sandro Stringari.
\newblock \emph{Bose--Einstein Condensation and Superfluidity}.
\newblock Oxford University Press, 2016.
}

\bibitem{pankov2005}
Alexander Pankov.
\newblock Periodic nonlinear Schr\"odinger equation with application to photonic crystals.
\newblock {\em Milan Journal of Mathematics}, 73:259--287, 2005.

\bibitem{pellacciPistoiaVairaVerzini2021}
{
Benedetta Pellacci, Angela Pistoia, Giusi Vaira, and Gianmaria Verzini.
\newblock Normalized concentrating solutions to nonlinear elliptic problems.
\newblock {\em Journal of Differential Equations}, 275:882--919, 2021.
}

\bibitem{pistoiaRamos2004}
{
Angela Pistoia and Miguel Ramos.
\newblock Locating the peaks of the least energy solutions to an elliptic system with Neumann boundary conditions.
\newblock {\em Journal of Differential Equations}, 201(1):160--176, 2004.
}

\bibitem{pistoiaSaldanaTavares2023}
{
Angela Pistoia, Alberto Salda\~na, and Hugo Tavares.
\newblock Existence of solutions to a slightly supercritical pure Neumann problem.
\newblock {\em SIAM Journal on Mathematical Analysis}, 55(4):3844--3887, 2023.
}

\bibitem{pistoiaSchiera2025}
{
Angela Pistoia and Delia Schiera.
\newblock On a critical Hamiltonian system with Neumann boundary conditions.
\newblock {\em Topological Methods in Nonlinear Analysis}, 66(1):187--213, 2025.
}






\bibitem{shenwang2013}
{
Yaotian Shen and Youjun Wang.
\newblock Soliton solutions for generalized quasilinear Schr\"odinger equations.
\newblock \emph{Nonlinear Anal.}, 80:194--201, 2013.
}

\bibitem{szulkin2009ground}
Andrzej Szulkin and Tobias Weth.
\newblock Ground state solutions for some indefinite variational problems.
\newblock {\em Journal of Functional Analysis}, 257(12):3802--3822, 2009.

\bibitem{szulkin2010method}
Andrzej Szulkin and Tobias Weth.
\newblock The method of the Nehari manifold.
\newblock {\em Handbook of nonconvex analysis and applications}, 597632, 2010.

\bibitem{willem1996minimax}
Michel Willem.
\newblock {\em Minimax Theorems}. Progress in Nonlinear Differential Equations and Their Applications, Birkh\"auser Boston Inc, Boston, 1996.

  
\end{thebibliography}
\end{document}